\documentclass[11pt,reqno]{amsart}
\usepackage{color}
\usepackage[english]{babel}
\usepackage{amsfonts}
\usepackage{amsmath}
\usepackage{amsthm}
\usepackage{amssymb}
\usepackage[misc]{ifsym}
\usepackage{bbm}
\usepackage{mathrsfs}
\usepackage{esint}
\usepackage{faktor}
\usepackage[utf8]{inputenc}
\usepackage{afterpage}
\usepackage[left=2.9cm,right=2.9cm,top=2.8cm,bottom=2.8cm]{geometry}
\usepackage{graphicx}
\usepackage{enumitem}
\usepackage[dvipsnames]{xcolor}
\usepackage[colorlinks=true,urlcolor=blue,citecolor=blue,linkcolor=blue,linktocpage,pdfpagelabels,bookmarksnumbered,bookmarksopen]{hyperref}

\newcommand{\N}{\mathbb{N}}

\newcommand{\R}{\mathbb{R}}
\newcommand{\A}{\mathcal{A}}

\newcommand{\D}{\mathcal{D}}

\newcommand{\dx}{\, {\rm d} x}

\newcommand{\dt}{\, {\rm d} t}

\newcommand{\eps}{\varepsilon}
\newcommand{\loc}{{\rm loc}}
\newcommand{\compact}{\stackrel{c}{\hookrightarrow}}
\renewcommand{\S}{\mathscr{S}}
\newcommand{\T}{\mathcal{T}}
\newcommand{\B}{\mathcal{B}}

\newtheorem{lemma}{Lemma}[section]
\newtheorem{thm}[lemma]{Theorem}

\theoremstyle{definition}

\newtheorem{rmk}[lemma]{Remark}

\numberwithin{equation}{section}

\begin{document}

\title[A $p$-Laplacian problem in $\R^N$ with singular, convective, critical reaction]{Existence and regularity for a $p$-Laplacian problem \\ in $\R^N$ with singular, convective, critical reaction}

\author[L. Baldelli]{Laura Baldelli}
\address[L. Baldelli]{IMAG, Departamento de Análisis Matemático, Universidad de Granada, Campus Fuentenueva, 18071 Granada, Spain}
\email{labaldelli@ugr.es}


\author[U. Guarnotta]{Umberto Guarnotta}
\address[U. Guarnotta]{Università degli Studi di Enna ``Kore'', Department of Engineering and Architecture, 94100 Enna, Italy}
\email{umberto.guarnotta@unikore.it}

\maketitle

\begin{abstract}
We prove an existence result for a $p$-Laplacian problem set in the whole Euclidean space and exhibiting a critical term perturbed by a singular, convective reaction. The approach used combines variational methods, truncation techniques, and concentration compactness arguments, together with set-valued analysis and fixed point theory. De Giorgi's technique, a priori gradient estimates, and nonlinear regularity theory are employed to get local $C^{1,\alpha}$ regularity of solutions, as well as their pointwise decay at infinity. The result is new even in the non-singular case, also for the Laplacian.
\end{abstract}

{
\let\thefootnote\relax
\footnote{\Letter \quad Corresponding author: Umberto Guarnotta (umberto.guarnotta@unikore.it).}
}
\setcounter{footnote}{0}

\begin{center}
\begin{minipage}{11cm}
\begin{small}
\tableofcontents
\end{small}
\end{minipage}
\end{center}

\section{Introduction}

In the present paper we consider the problem
\begin{equation}
\label{prob}
\tag{${\rm P}_\lambda$}
\left\{
\begin{alignedat}{2}
-\Delta_p u &=\lambda w(x)f(u,\nabla u) +  u^{p^*-1} \quad &&\mbox{in} \;\; \R^N, \\
u &> 0 \quad &&\mbox{in} \;\; \R^N, \\
u(x) &\to 0 \quad &&\mbox{as} \;\; |x|\to+\infty,
\end{alignedat}
\right.
\end{equation}
where $1<p<N$ (whence $N\geq 2$), $\Delta_p$ is the $p$-Laplacian operator, $\lambda>0$, and $p^*:=\frac{Np}{N-p}$ is the Sobolev critical exponent. We assume the following hypotheses:
\begin{enumerate}[label={$({\rm H}_f)$}]
\item \label{hypf}
The function $f:(0,+\infty) \times \R^N \to (0,+\infty)$ is a continuous function obeying
\begin{equation*}
c_1 s^{-\gamma} \leq f(s,\xi)\leq c_2(s^{-\gamma}+|\xi|^{r-1}) \quad \forall (s,\xi) \in (0,+\infty)\times \R^N,
\end{equation*}
for some $0<\gamma<1<r<p$ and $c_1,c_2>0$.
\end{enumerate}
\begin{enumerate}[label={$({\rm H}_w)$}]
\item \label{hypw}
The function $w:\R^N\to(0,+\infty)$ belongs to $L^1(\R^N) \cap L^\infty(\R^N)$ and satisfies the following conditions:
\begin{equation}
\label{weightdecay}
\mbox{there exist} \;\; c_3,R>0 \;\; \mbox{and} \;\; l> N+\gamma \, \frac{N-p}{p-1} \;\; \mbox{such that} \;\; w(x) \leq c_3|x|^{-l} \quad \forall x\in B_R^e;
\end{equation}
\begin{equation}
\label{weightbelow}
\mbox{there exist} \;\; x_0\in\R^N \;\; \mbox{and} \;\; \varrho,\omega>0 \;\; \mbox{such that} \;\; \inf_{B_\varrho(x_0)} w \geq \omega.
\end{equation}
\end{enumerate}

Describing, forecasting, and controlling the evolution of a variety of phenomena in physics, chemistry, finance, biology, ecology, medicine, sociology, and industrial activity cannot be done without taking into account nonlinear equations. For what concerns the principal part of the operator, the $p$-Laplacian arises in the theory of quasi-regular and quasi-conformal mappings, as well as it provides a mathematical model of non-Newtonian fluids ($1<p<2$ represents pseudoplastic fluids, such as lava, while $p>2$ describes dilating fluids, such as blood); see \cite{am}. On the other hand, reaction terms exhibiting singular nonlinearities are important in natural sciences, such as in the study of heat conduction
in electrically conducting materials \cite{fm} and for chemical heterogeneous catalysts \cite{p}. Also convective elliptic problems naturally arise from applicative questions, such as optimal stochastic control problems (cf. \cite[p.241]{gr8}). Other models exploiting singular and convective elliptic problems can be found in the monography \cite{gr8}.

The behavior of equations in which a critical term is perturbed with a lower-order term is studied in many model problems, such as the Yamabe problem, the problem of searching an extremal function for the isoperimetric inequality, as well as the existence of non-minimal solutions of the Yang-Mills equation (see \cite{bn} and the references therein): this is one of the motivations which inspired the present work.

The aim of the paper is to prove the following result.
\begin{thm}
\label{mainthm}
Suppose \ref{hypf}--\ref{hypw}. Then there exists $\Lambda>0$ such that, for any $\lambda\in(0,\Lambda)$, problem \eqref{prob} admits a weak solution in $\D^{1,p}_0(\R^N)\cap L^\infty(\R^N)\cap C^{1,\alpha}_\loc(\R^N)$, for some $\alpha\in(0,1]$.
\end{thm}

Problem \eqref{prob} exhibits several features:
\begin{itemize}
\item the perturbation $f$ is singular, i.e., it blows up when the solution vanishes;
\item $f$ encompasses also convection terms, that is, depending on the gradient of the solution;
\item the `dominating' reaction term has critical growth;
\item the setting is the whole $\R^N$;
\item pointwise decay (at infinity) of the solutions is required.
\end{itemize}

\subsection{Comparison with previous results}
The main motivation behind the analysis of \eqref{prob} is that it mixes variational problems with double lack of compactness, since the loss of compactness of Sobolev's embedding occurs due to both the presence of the Sobolev critical exponent and the unboundedness of the domain, with non-variational problems, since convection terms destroy the variational structure. The multi-faced aspect of the problem suggests that it may be of interest to summarize the state of art of elliptic problems exhibiting, either separately or jointly, critical, singular, and convective reaction terms.

The pioneering paper by Brezis and Nirenberg \cite{bn} for the Laplacian in a bounded domain paved the way for critical problems in the last thirty years.
Later, existence and multiplicity of infinitely many solutions were obtained for the $p$-Laplacian in the critical case by Garcia Azorero and Peral Alonso \cite{ap} in bounded domains and by Huang \cite{h99} in $\R^N$, applying the mountain pass theorem and the theory of Krasnosel'ski\v{i} genus. More general operators were investigated by several authors: see \cite{ByF2, BFsc, prr}. In all of these papers the concentration compactness principles by Lions, Ben-naoum et al. \cite{L3,BNTW} play a crucial role in recovering compactness.

In the same years, Lazer, McKenna et al. \cite{LM,CLM} gave the decisive boost to the thereafter florid line of research of singular problems. Often, existence of solutions to $p$-Laplacian singular problems is obtained either by combining variational methods with sub-super-solution and truncation techniques (see, e.g., \cite{pesi}), or by regularizing the singular terms and using a priori estimates to recover compactness (see the appendix of \cite{CST}).

The presence of convection terms destroys the variational structure of even more basic problems than \eqref{prob}. For this reason, methods involving a priori estimates, Liouville-type theorems, and degree arguments \cite{BFNA,Ruiz,gisp} are widely employed in these situations when working on bounded domains; on the contrary, little is known in the entire Euclidean setting, and monotonicity techniques (as comparison and sub-super-solution theorems) are essential ingredients of the proofs: see \cite{fmt}.

Concerning the interaction of singular and convective nonlinearities, we address the reader to the survey \cite{glm} (see also \cite{glm2} for systems), which contains a rich bibliography on this topic. However, to the best of our knowledge, \cite{GG1,GMMou} are the only papers involving singular quasilinear elliptic problems in the whole space and with convective terms, regarding equations and systems, respectively. 

Finally, about singular and critical reactions, we refer to \cite{gst} for the case of bounded domains and \cite{mps} for the nonlocal setting; see also \cite{gksr} for the fractional $p$-Laplacian. As far as we know, there are no results in unbounded domains.

Theorem \ref{mainthm} is a first attempt at dealing with critical, convective, and singular elliptic equations in the entire $\R^N$. Moreover, up to our knowledge, our result is new even in the non-singular case, also for the Laplacian ($p=2$).

\subsection{Sketch of the proof}
Let us briefly summarize the proof of Theorem \ref{mainthm}, exposing the main techniques adopted in the paper.

First of all, we truncate and freeze the perturbation $f$, in order to cast the problem into a classical variational framework: indeed, freezing the convection terms (i.e., keeping them fixed) allows to get rid of them, which makes the problem to fall into a variational setting, while truncation guarantees $C^1$ regularity of the associated energy functional. After truncating and freezing, we obtain the problem
\begin{equation*}
\tag{${\rm \hat{P}}_\lambda$}
-\Delta_p u = \lambda a(x,u) + u_+^{p^*-1} \quad \mbox{in} \;\; \R^N,
\end{equation*}
where
\begin{equation*}
a(x,s):=w(x)f(\max\{s,\underline{u}_\lambda(x)\},\nabla v(x)) \quad \forall (x,s)\in \R^N\times\R,
\end{equation*}
being $\underline{u}_\lambda$ a suitable function (see \eqref{usub}) and $v\in\D^{1,p}_0(\R^N)$ fixed.

The central part of the paper is devoted to construct a solution $u\in\D^{1,p}_0(\R^N)$ to \eqref{varprob} by applying the mountain pass theorem to the energy functional $J$ associated with \eqref{varprob} (see Theorem \ref{existencefinal}). This can be done in three steps:
\begin{enumerate}
\item detecting, through concentration compactness principles, a particular energy level $\hat{c}$ (see \eqref{hatc}) under which compactness of $J$ is recovered, that is, proving that $J$ satisfies the Palais-Smale condition below $\hat{c}$ (Lemma \ref{PS});
\item showing that, for small $\lambda$'s, $J$ satisfies the mountain pass geometry, `joining' the origin with a Talenti's function (Lemma \ref{mountainpassgeometry});
\item ensuring that the mountain pass level lies below the `critical' Palais-Smale level $\hat{c}$, provided $\lambda$ is small enough (Lemma \ref{talentihatc}).
\end{enumerate}

Then we come back to the original problem. Since the truncation was performed at the level of a sub-solution $\underline{u}_\lambda$ (cf. Lemma \ref{subsol} and \eqref{usub}), any solution to \eqref{varprob} stays above $\underline{u}_\lambda$ (see Remark \ref{subcomparisonrmk}). It remains to unfreeze the convection terms: this is achieved by using set-valued analysis and fixed point theory. More specifically, the set-valued function $\S$ associating to each function $v\in\D^{1,p}_0(\R^N)$ the set of solutions to \eqref{varprob} having `low' energy (see \eqref{Sdef}) is compact (Lemma \ref{Scompact}) and, for small values of $\lambda$, lower semi-continuous (Lemma \ref{lsc}), so its selection $\T$ obtained by minimality (see \eqref{Tdef}) inherits compactness and continuity; thus, Schauder's theorem guarantees the existence of a fixed point of $\T$ (Theorem \ref{exsol}), that is, a solution to \eqref{prob}, provided a pointwise decay at infinity is ensured.

In order to conclude, both local $C^{1,\alpha}$ regularity and pointwise decay for solutions $u$ to \eqref{prob} lying in any energy level below $\hat{c}$ are proved (see Theorem \ref{regularity}). To do this, a global $L^\infty$ estimate is ensured via both De Giorgi's thechnique and the uniform equi-integrability provided by the concentration compactness principles (cf. Lemma \ref{ccplemma}); then a quantitative local $L^\infty$ gradient estimate guarantees local $C^{1,\alpha}$ regularity of the solution, as well as a high global summability on the weighted gradient term $w|\nabla u|^{r-1}$, which in turn ensures a global $L^\infty$ gradient estimate; finally, this information allows a comparison with a radial function which decays at infinity.

Before analyzing more technical aspects, it is worth mentioning that we provide quantitative estimates on the threshold $\Lambda$ in Theorem \ref{mainthm}: see \eqref{smallness}, \eqref{smallness2}, \eqref{smallness3}, and \eqref{lsccond2}.

\subsection{Technical issues}

Now we discuss some technical details, explicitly pointed out to emphasize some delicate, and somehow innovative, points in the proofs along the paper.

Firstly, we observe that the concentration compactness argument, carried out in a general form in Lemma \ref{ccplemma}, is crucial in this setting. Indeed, it is applied four times: (i) to ensure the Palais-Smale condition on some energy levels, as customary when dealing with critical problems (cf. \cite{ap,BBF,ByF2}); (ii) to get compactness of the set-valued function $\S$, usually recovered by working in the $C^1$ topology via Ascoli-Arzelà's theorem (cf. \cite{LMZ,GMMot,GMar}; see also \cite{BFNA}); (iii) to guarantee the lower semi-continuity of $\S$, which heavily relies on the compactness of the sub-level sets of energy functional; (iv) to provide $L^\infty$ estimates uniform with respect to the solution, using an equi-uniform integrability information which is, in general, not available without restrictions on the energy levels (cf., e.g., \cite{CGL}).

Secondly, it is worth pointing out that compactness and lower semi-continuity of $\S$ require additional work, compared to \cite{LMZ,GMMot,GMar}, and this is due not only to the concentration compactness issues mentioned above. About compactness, working in Beppo Levi spaces causes the loss of a.\,e.\,convergence of gradient terms $(\nabla v_n)$ with fixed $(v_n)$, forcing the use of a monotonicity property, as the $(S_+)$ property for the $p$-Laplacian in $\D^{1,p}_0(\R^N)$ (cf. \cite{MMM}). Concerning lower semi-continuity, the super-linear growth of the reaction term requires the joint usage of a scaling argument and a fine recursive estimate to get suitable energy bounds.

Lastly, we spend a few words about regularity and decay of solutions $u$ to \eqref{prob}. Unlike the existence result, which requires only $w\underline{u}_\lambda^{-\gamma}\in L^{(p^*)'}(\R^N)$, we need a decay on $\underline{u}_\lambda$ to get $C^{1,\alpha}$ local estimates. Indeed, the decay of $\underline{u}_\lambda$ guarantees $w\underline{u}_\lambda^{-\gamma}\in L^\infty(\R^N)$, producing local $L^\infty$ gradient estimates which imply $C^{1,\alpha}$ local regularity. The decay of $\underline{u}_\lambda$ is a consequence of the decays of the fundamental solution of the $p$-Laplacian and $w$; the latter is also used to refine $L^\infty$ bounds of $\nabla u$, deducing global estimates from local ones, and to ensure the pointwise decay of $u$.

\subsection{Structure of the paper} In Section \ref{prel} we give a few classical definitions and state some basic results, such as the concentration compactness principles, the mountain pass theorem, and a fixed point theorem, together with minor lemmas which will be useful in the sequel. Section \ref{trfr} is devoted to the study of a truncated and frozen problem, which leads to an existence result. In Section \ref{unfr} the unfreezing the convection term is carried out via set-valued analysis. Finally, Section \ref{reg} contains the last part of the proof of Theorem \ref{mainthm}, concerning regularity and decay of solutions.

\section{Preliminaries}\label{prel}
\subsection{Notations}
We indicate with $B_r(x)$ the $\R^N$-ball of center $x\in\R^N$ and radius $r>0$, omitting $x$ when it is the origin. The symbols $\overline{B}$, $\partial B$, $B^e$ stand, respectively, for the closure, the boundary, and the exterior of the ball $B$. Given any $A\subseteq \R^N$, we write $\chi_A$ to indicate the characteristic function of $A$. For any $N$-dimensional Lebesgue measurable set $\Omega$, by $|\Omega|$ we mean its $N$-dimensional Lebesgue measure.

Given a real-valued function $\varphi$, we indicate its positive (resp., negative) part with $\varphi_+:=\max\{\varphi,0\}$ (resp., $\varphi_-:=\max\{-\varphi,0\}$). We abbreviate with $\{u>v\}$ the set $\{x\in\R^N: \, u(x)>v(x)\}$, and similarly for $\{u<v\}$, etc.

We indicate with $X^*$ the dual of a Banach space $X$, while $\langle\cdot,\cdot\rangle$ stand for the duality brackets. Given two Banach spaces $X,Y$, the continuous embedding of $X$ into $Y$ is indicated by $X\hookrightarrow Y$; if the embedding is compact, we write $X\compact Y$. If a sequence $(u_n)$ strongly converges to $u$ we write $u_n\to u$; if the convergence is in weak sense, we use $u_n\rightharpoonup u$. The letter $S$ denotes the Sobolev constant; see the next subsection for details.

Let $M(\R^N,\R)$ be the space of all finite signed Radon measures. Concerning convergence of measures $(\mu_n)\subseteq M(\R^N,\R)$, we write $\mu_n\stackrel{*}{\rightharpoonup}\mu$ and $\mu_n\rightharpoonup\mu$ to signify tight and weak convergence, respectively (the definitions are given in the next subsection). The symbol $\delta_x$ indicates the Dirac delta of mass 1 concentrated at $x\in\R^N$.

The letter $C$ denotes a positive constant which may change its value at each passage; subscripts on $C$ emphasize its dependence from the specified parameters. For the sake of readability, we also write `in $\R^N$' instead of `a.e.\,in $\R^N$'.
\subsection{The functional setting}
We denote by $C^\infty_c(\R^N)$ the space of the compactly supported test functions on $\R^N$, while $C^{1,\alpha}_\loc(\R^N)$, being $\alpha\in(0,1]$, denotes the space of continuously differentiable functions whose gradient is locally $\alpha$-H\"older continuous.

Given any measurable set $\Omega\subseteq \R^N$ and $q\in[1,+\infty]$, $L^q(\Omega)$ stands for the standard Lebesgue space, whose norm will be indicated with $\|\cdot\|_{L^q(\Omega)}$, or simply $\|\cdot\|_q$ when $\Omega=\R^N$. For $p\in(1,N)$, we will also make use of the Beppo Levi space $\D^{1,p}_0(\R^N)$, which is the closure of $C^\infty_c(\R^N)$ with respect to the norm
$$\|u\|_{\D^{1,p}_0(\R^N)}:=\|\nabla u\|_p.$$
Beppo Levi spaces are reflexive, separable Banach spaces; we indicate with $\D^{-1,p'}(\R^N)$ the dual of $\D^{1,p}_0(\R^N)$. Sobolev's theorem ensures that $\D^{1,p}_0(\R^N) \hookrightarrow L^{p^*}(\R^N)$; the best constant $c$ in the Sobolev inequality $\|u\|_{L^{p^*}(\R^N)}\leq c \|u\|_{\D^{1,p}_0(\R^N)}$ is $S^{-1/p}$, being
$$ S:= \inf_{u\in \D^{1,p}_0(\R^N)\setminus\{0\}} \frac{\|\nabla u\|_p^p}{\|u\|_{p^*}^p}. $$
According to Sobolev's theorem, one has
$$\D^{1,p}_0(\R^N) = \left\{u\in L^{p^*}(\R^N): \, |\nabla u|\in L^p(\R^N)\right\}.$$ Incidentally, we recall that $\D^{1,p}_0(\R^N) \compact L^q(\Omega)$ for all bounded, measurable sets $\Omega$ and all $q\in[1,p^*)$ (this is a consequence of Rellich-Kondrachov's theorem; cf. \cite[Theorem 9.16]{B}).

A sequence of measures $(\mu_n)\subseteq M(\R^N,\R)$ converges tightly to a measure $\mu$, written as $\mu_n \stackrel{*}{\rightharpoonup} \mu$, if
\begin{equation}
\label{measconv}
\int_{\R^N} f \, {\rm d}\mu_n \to \int_{\R^N} f \, {\rm d}\mu \quad \mbox{for all} \;\; f\in C_b(\R^N),
\end{equation}
where $C_b(\R^N)$ is the space of the bounded, continuous functions on $\R^N$. On the other hand, $(\mu_n)\subseteq M(\R^N,\R)$ is said to converge weakly to $\mu$, written as $\mu_n \rightharpoonup \mu$, if \eqref{measconv} holds for all $f\in C_0(\R^N)$, being $C_0(\R^N)$ the space of the continuous functions that vanish at infinity. Since $C_0(\R^N)\subseteq C_b(\R^N)$, tight convergence implies weak convergence. Moreover, if $(\mu_n)\subseteq M(\R^N,\R)$ is bounded, then (up to sub-sequences) $\mu_n\rightharpoonup \mu$ for some $\mu\in M(\R^N,\R)$: see \cite[Proposition 1.202]{FL}. Notice that weak convergence is the `natural' convergence in the space $M(\R^N,\R)$, since $M(\R^N,\R)=(C_0(\R^N))'$. It is worth pointing out that tight convergence can be seen as non-concentration at infinity: see \cite{BBF, BBFS}.

Let $(X,\|\cdot\|_X)$ be a Banach space and $J$ be a functional of class $C^1$ (hereafter indicated as $J\in C^1(X)$). A sequence $(u_n)\subseteq X$ is a Palais-Smale sequence of level $c\in\R$ if $J(u_n) \to c$ in $\R$ and $J'(u_n)\to 0$ in $X^*$. If each Palais-Smale sequence of level $c$ admits a strongly convergent sub-sequence, then $J$ is said to satisfy the Palais-Smale condition at level $c$; briefly, $J$ satisfies ${\rm (PS)}_c$.

A poset (that is, a partially ordered set) $(A,\leq)$ is said to be downward directed if for any $a,b\in A$ there exists $c\in A$ such that $c\leq a$ and $c\leq b$. We recall that, if $a$ is a minimal element of the downward directed poset $A$, then $a=\min A$: indeed, since $A$ is downward directed, for any $b\in A$ there exists $c\in A$ such that $c\leq a$ and $c\leq b$, and minimality of $a$ forces $c=a$, so that $c\leq b$ for all $b\in A$, i.e., $c=\min A$.

Let $(X,d_X),(Y,d_Y)$ be two metric spaces. A set-valued function $\S:X\to2^Y$ is said to be lower semi-continuous if, for any $x_n\to x$ in $X$ and $y\in\S(x)$, there exists $(y_n)\subseteq Y$ such that $y_n\to y\in Y$ and $y_n\in\S(x_n)$ for all $n\in\N$; it is said to be compact if, for any bounded $K\subseteq X$, the set $\S(K)$ is relatively compact in $Y$.

\subsection{Some tools}
First of all, we recall a simple result concerning weak convergence of the positive part of functions; although it is folklore, we make its proof.
\begin{lemma}
\label{posparts}
Let $1<p<N$ and $(u_n)\subseteq\D^{1,p}_0(\R^N)$, $u\in\D^{1,p}_0(\R^N)$ be such that $u_n \rightharpoonup u$ in $\D^{1,p}_0(\R^N)$. Then $(u_n)_+ \rightharpoonup u_+$ in $\D^{1,p}_0(\R^N)$.
\end{lemma}
\begin{proof}
Fix any $R>0$. Since $\D^{1,p}_0(\R^N)\compact L^p(B_R)$, we have $u_n\to u$ in $L^p(B_R)$. Then, up to sub-sequences, $u_n\to u$ in $B_R$ and there exists $U\in L^p(B_R)$ such that $|u_n|\leq U$ in $B_R$ for all $n\in\N$ (see \cite[Theorem 4.9]{B}). Thus, Lebesgue's dominated convergence theorem and the continuity of the real function $t\mapsto t_+$ imply $(u_n)_+\to u_+$ in $L^p(B_R)$.

We observe that $(u_n)$ is bounded in $\D^{1,p}_0(\R^N)$. Stampacchia's lemma (cf., e.g., \cite[Theorem 4.4]{EG}) guarantees that $\nabla (v_+)=\chi_{\{v>0\}}\nabla v$ for all $v\in \D^{1,p}_0(\R^N)$, and this implies that $((u_n)_+)$ is bounded in $\D^{1,p}_0(\R^N)$. By reflexivity, there exists $v\in\D^{1,p}_0(\R^N)$ such that $(u_n)_+ \rightharpoonup v$ in $\D^{1,p}_0(\R^N)$. Reasoning as above, we infer $(u_n)_+ \to v$ in $L^p(B_R)$, forcing $v=u_+$. Arbitrariness of $R$ gives $(u_n)_+ \rightharpoonup u_+$ in $\D^{1,p}_0(\R^N)$.
\end{proof}

Let us introduce two lemmas which are useful to handle concentration of compactness at points and at infinity, respectively.
\begin{lemma}[Lions, {\cite[Lemma I.1]{L3}}]
\label{lions}
Let $1\leq p<N$. Suppose $(u_n) \subseteq \D^{1,p}_0(\R^N)$ to be such that $u_n \rightharpoonup u$ in $\D^{1,p}_0(\R^N)$, and both $|\nabla u_n|^p \rightharpoonup \mu$, $|u_n|^{p^*} \stackrel{*}{\rightharpoonup} \nu$ in the sense of measures, for some $u\in\D^{1,p}_0(\R^N)$ and $\mu,\nu$ bounded non-negative measures on $\R^N$. Then there exist some at most countable set $\A$, a family of distinct points $(x_j)_{j\in\A}\subseteq \R^N$, and two families of numbers $(\nu_j)_{j\in\A}, (\mu_j)_{j\in\A}\subseteq (0,+\infty)$ fulfilling
\begin{equation}
\label{ineqmeasures}
\nu=|u|^{p^*}+\sum_{j\in\A} \nu_j \delta_{x_j},\qquad
\mu\geq |\nabla u|^p+\sum_{j\in\A} \mu_j \delta_{x_j},\qquad
S\nu_j^{p/p^*} \leq \mu_j \quad \mbox{for all} \;\; j\in\A.
\end{equation}
\end{lemma}

Note that Lemma \ref{lions} requires the tight convergence of the measures involving the critical Sobolev exponent, but the proof of this condition reveals to be rather difficult and technical. Thus, Ben-Naoum et al. established a version of the Lemma \ref{lions} known as {\it escape to infinity principle}, where the concentration at infinity is enclosed in the parameters $\nu_\infty$ and $\mu_\infty$.

\begin{lemma}[Ben-Naoum, Troestler, Willem, {\cite[Lemma 3.3]{BNTW}}]
\label{bennaoum}
Let $1\leq p<N$. Suppose that $(u_n) \subseteq \D^{1,p}_0(\R^N)$ is bounded and define
$$ \nu_\infty := \lim_{R\to+\infty} \limsup_{n\to\infty} \int_{B_R^e} |u_n|^{p^*} \dx, \quad \mu_\infty := \lim_{R\to+\infty} \limsup_{n\to\infty} \int_{B_R^e} |\nabla u_n|^p \dx.$$
Then, it holds $S\nu_\infty^{p/p^*} \leq \mu_\infty$ and
\begin{equation*}
\begin{split}
\limsup_{n\to\infty} \int_{\R^N} |u_n|^{p^*} \dx = \int_{\R^N} \, {\rm d}\nu + \nu_\infty, \qquad
\limsup_{n\to\infty} \int_{\R^N} |\nabla u_n|^p \dx = \int_{\R^N} \, {\rm d}\mu + \mu_\infty, 
\end{split}
\end{equation*}
where $\nu,\mu$ are as in Lemma \ref{lions}.
\end{lemma}

We will use Lemmas \ref{lions}--\ref{bennaoum} to avoid concentration both at points, i.e. $\nu_j=\mu_j=0$ for all $j\in\A$, and at
infinity, i.e. $\nu_\infty=\mu_\infty=0$.

\begin{rmk}
\label{nujmuj}
Although not explicitly stated in \cite[Lemma I.1]{L3}, for all $j\in\A$ one has $\nu_j=\nu(\{x_j\})$ and one can assume $\mu_j=\mu(\{x_j\})$ in Lemma \ref{lions} (see, e.g., \cite[Theorem 2.5]{MS}). Indeed, for all $\eps>0$, \eqref{ineqmeasures} implies
$$ \nu(B_\eps(x_j))=\int_{B_\eps(x_j)} |u|^{p^*} \dx +\sum_{i\in\A} \nu_i \delta_{x_i}(B_\eps(x_j)) \quad \mbox{for all} \;\; j\in\A, $$
so letting $\eps\to 0$ produces $\nu(\{x_j\})=\nu_j$ for all $j\in\A$. On the other hand, it is readily seen that $\mu\geq \mu(\{x_j\})\delta_{x_j}$ for all $j\in\A$. Since $\{|\nabla u|^p\}\cup\{\delta_{x_j}: \, j\in\A\}$ is a set consisting of pairwise mutually singular measures, as well as \eqref{ineqmeasures} ensures $\mu\geq |\nabla u|^p$, one has
$$ \mu\geq |\nabla u|^p + \sum_{j\in\A} \mu(\{x_j\})\delta_{x_j}. $$
Using \eqref{ineqmeasures} again, one has $\mu(\{x_j\})\geq \mu_j \geq S\nu_j^{p/p^*}$. Thus, replacing $\mu_j$ with $\mu(\{x_j\})$ for all $j\in\A$ leaves \eqref{ineqmeasures} unaltered.
\end{rmk}

Now we state the refined versions of two pivotal theorems in Nonlinear Analysis.
\begin{thm}[Ambrosetti, Rabinowitz; cf. {\cite[Theorem 5.40]{MMP}}]
\label{mountainpass}
Let $(X,\|\cdot\|_X)$ be a Banach space and $J\in C^1(X)$. Let $u_0,u_1\in X$, $\rho>0$ such that the `mountain pass geometry' is fulfilled, i.e.,
$$ \max\{J(u_0),J(u_1)\} < \inf_{\partial B_\rho(u_0)} J =: b \quad \mbox{and} \quad \|u_1-u_0\|_X > \rho.$$
Set
$$\Phi:= \{\phi\in C^0([0,1];X): \, \phi(0)=u_0, \, \phi(1)=u_1\}, \quad c_M:= \inf_{\phi\in\Phi} \sup_{t\in[0,1]} J(\phi(t)).$$
If $J$ satisfies ${\rm (PS)}_{c_M}$, then $c_M\geq b$ and there exists $u\in X$ such that both $J(u)=c_M$ and $J'(u)=0$. Moreover, if $c_M=b$, then $u$ can be taken on $\partial B_\rho(u_0)$.
\end{thm}

\begin{thm}[Schauder; cf. {\cite[Theorem 6.3.2 p.119]{GD}}]
\label{schauder}
Let $K$ be a non-empty bounded convex subset of a normed linear space $E$, and let $T:K\to K$ be a compact map. Then $T$ has a fixed point.
\end{thm}

We will need also the following result about boundedness of sequences defined by recursion.

\begin{lemma}
\label{reclemma}
Let $(b_k)_{k=0}^\infty\subseteq [0,+\infty)$ satisfy, for some $c>0$ and $K,\alpha>1$, the recursion
\begin{equation*}
b_k \leq c + Kb_{k-1}^{\alpha} \quad \mbox{for all}\;\; k\in\N.
\end{equation*}
If
\begin{equation}
\label{smallnessconds}
K b_0^{\alpha-1}\leq\frac{1}{2} \quad \mbox{and} \quad Kc^{\alpha-1}<2^{-\alpha},
\end{equation}
then $(b_k)$ is bounded.
\end{lemma}
\begin{proof}
By iteration we get
\begin{equation*}
\begin{split}
b_k &\leq c + Kb_{k-1}^\alpha \leq c + K(c+Kb_{k-2}^\alpha)^\alpha \leq c+2^\alpha K(c^\alpha+K^\alpha b_{k-2}^{\alpha^2}) \\
&\leq c+2^\alpha K c^\alpha + 2^\alpha K^{1+\alpha}(c+Kb_{k-3}^\alpha)^{\alpha^2} \leq c+2^\alpha K c^\alpha + 2^{\alpha+\alpha^2}K^{1+\alpha}(c^{\alpha^2}+K^{\alpha^2}b_{k-3}^{\alpha^3}) \\
&\leq \ldots \leq c + \sum_{j=1}^{k-1} 2^{\sum_{i=1}^j\alpha^i} K^{\sum_{i=0}^{j-1}\alpha^i} c^{\alpha^j} + 2^{\sum_{i=1}^{k-1}\alpha^i} K^{\sum_{i=0}^{k-1}\alpha^i} b_0^{\alpha^k} \\
&= \sum_{j=0}^{k-1} 2^{\frac{\alpha^{j+1}-\alpha}{\alpha-1}} K^{\frac{\alpha^j-1}{\alpha-1}} c^{\alpha^j} + 2^{\frac{\alpha^k-\alpha}{\alpha-1}} K^{\frac{\alpha^k-1}{\alpha-1}} b_0^{\alpha^k} \leq \sum_{j=0}^{k-1} (2^{\frac{\alpha}{\alpha-1}} K^{\frac{1}{\alpha-1}} c)^{\alpha^j} + (2^{\frac{1}{\alpha-1}} K^{\frac{1}{\alpha-1}} b_0)^{\alpha^k}.
\end{split}
\end{equation*}
The conclusion then follows by \eqref{smallnessconds} observing that, for all $k\in\N$,
\begin{equation*}
b_k \leq \sum_{j=0}^\infty (2^{\frac{\alpha}{\alpha-1}} K^{\frac{1}{\alpha-1}} c)^{\alpha^j} + 1 <+\infty,
\end{equation*}
since $\sum_{j=0}^\infty q^{\alpha^j}$ is convergent for all $q\in(0,1)$.
\end{proof}

We conclude this section by proving the existence of a sub-solution to \eqref{prob} (see Remark \ref{subcomparisonrmk} below). We premit a straightforward adaptation of the weak comparison principle (cf. \cite[Theorem 3.4.1]{PS}) to the setting of Beppo Levi spaces.
\begin{lemma}
\label{weakcomp}
Let $u,v\in\D^{1,p}_0(\R^N)$ satisfy
\begin{equation}
\label{weakcomptest}
\langle -\Delta_p u,\varphi \rangle \leq \langle -\Delta_p v,\varphi\rangle
\end{equation}
for all $\varphi\in\D^{1,p}_0(\R^N)$ such that $\varphi\geq 0$ in $\R^N$ and $\varphi\equiv 0$ on $\{u\geq v\}$. Then $u\leq v$ in $\R^N$.
\end{lemma}
\begin{proof}
Testing \eqref{weakcomptest} with $(u-v)_+\in\D^{1,p}_0(\R^N)$, besides recalling Stampacchia's lemma (cf., e.g., \cite[Theorem 4.4]{EG}), yields
$$ \int_{\{u>v\}} \left(|\nabla u|^{p-2} \nabla u-|\nabla v|^{p-2}\nabla v\right) (\nabla u-\nabla v) \dx \leq 0. $$
Hence, \cite[Lemma 2.1]{D} entails $\nabla (u-v)_+ = 0$ in $\R^N$, which implies $(u-v)_+=0$ in $\R^N$, since $(u-v)_+\in\D^{1,p}_0(\R^N)$.
\end{proof}

\begin{lemma}\label{subsol}
Let $w$ satisfy \ref{hypw} and let $\gamma\in(0,1)$. Then there exists a unique $u\in C^{1,\alpha}_\loc(\R^N)$ solution to
\begin{equation}
\label{subprob}
\left\{ \begin{alignedat}{2}
-\Delta_p u &= w(x) u^{-\gamma} \quad &&\mbox{in} \;\; \R^N, \\
u &> 0 \quad &&\mbox{in} \;\; \R^N, \\
u(x) &\to 0 \quad &&\mbox{as} \;\; |x|\to+\infty.
\end{alignedat}
\right.
\end{equation}
Moreover, $wu^{-\gamma}\in L^1(\R^N)\cap L^\infty(\R^N)$.
\end{lemma}
\begin{proof}
For all $n\in\N$, consider the regularized problems
\begin{equation}
\label{subregular}
\tag{${\rm\underline{P}}_n$}
-\Delta_p u_n = w(x) \left((u_n)_++\frac{1}{n}\right)^{-\gamma} \quad \mbox{in} \;\; \R^N.
\end{equation}
Fix any $n\in\N$. Direct methods of Calculus of Variations (see \cite[Theorem I.1.2]{S}) ensure that there exists $u_n\in\D^{1,p}_0(\R^N)$ solution to \eqref{subregular}. Notice that
$$0 \leq w \left((u_n)_++\frac{1}{n}\right)^{-\gamma} \leq n^\gamma w \in L^\infty(\R^N), $$
so $u_n\in C^{1,\alpha}_\loc(\R^N)$ according to nonlinear regularity theory \cite[Corollary p.830]{DB}. Moreover, by the strong maximum principle (cfr. \cite[Theorem 1.1.1]{PS}), $u_n>0$ in $\R^N$.

Testing \eqref{subregular} with $u_n$, besides using H\"older's and Sobolev's inequalities, yields
\begin{equation*}
\|\nabla u_n\|_p^p = \int_{\R^N} w\left(u_n+\frac{1}{n}\right)^{-\gamma} u_n \dx \leq \int_{\R^N} wu_n^{1-\gamma} \dx \leq \|w\|_\zeta \|u_n\|_{p^*}^{1-\gamma} \leq S^{-\frac{1-\gamma}{p}}\|w\|_\zeta \|\nabla u_n\|_p^{1-\gamma},
\end{equation*}
being $\zeta>1$ such that $\frac{1}{\zeta}+\frac{1-\gamma}{p^*}=1$. We deduce $\|\nabla u_n\|_p \leq (S^{-\frac{1-\gamma}{p}}\|w\|_\zeta)^{\frac{1}{p-1+\gamma}}$, so that $(u_n)$ is bounded in $\D^{1,p}_0(\R^N)$. By reflexivity, $u_n\rightharpoonup u$ in $\D^{1,p}_0(\R^N)$ for some $u\in\D^{1,p}_0(\R^N)$, up to sub-sequences.

Observing that $u_n+\frac{1}{n}>u_{n+1}+\frac{1}{n+1}$ on $\{u_n>u_{n+1}\}$ we get the weak inequality
$$-\Delta_p u_n = \left(u_n+\frac{1}{n}\right)^{-\gamma} \leq \left(u_{n+1}+\frac{1}{n+1}\right)^{-\gamma} = -\Delta_p u_{n+1} \quad \mbox{on} \;\; \{u_n>u_{n+1}\}. $$
According to Lemma \ref{weakcomp}, it turns out that $u_n\leq u_{n+1}$ in $\R^N$. Thus, we can define a measurable function $\tilde{u}$ such that $u_n \nearrow \tilde{u}$ in $\R^N$.

We show that $u=\tilde u$ in $\mathbb R^N$. For all $k\in\N$, $\D^{1,p}_0(\R^N)\compact L^p(B_k)$, so $u_n \to u$ in $L^p(B_k)$ and $u_n\to u$ in $B_k$, up to sub-sequences. A diagonal argument ensures that $u_n\to u$ in $\R^N$, whence $u=\tilde{u}$ in $\R^N$ (see, e.g., \cite[p.3044]{GG1} for details).

Now we prove that we can pass to the limit in the weak formulation of \eqref{subregular}. Taking any $\varphi\in\D^{1,p}_0(\R^N)$, we have
\begin{equation}
\label{sublimit1}
\lim_{n\to\infty} \int_{\R^N} |\nabla u_n|^{p-2} \nabla u_n \nabla \varphi \dx = \int_{\R^N} |\nabla u|^{p-2} \nabla u \nabla \varphi \dx,
\end{equation}
according to \cite[Proposition 1]{GG1} applied to the sequence $(|\nabla u_n|^{p-2}\nabla u_n)\subseteq L^{p'}(\R^N)$. On the other hand, splitting $\varphi=\varphi_+-\varphi_-$, Beppo Levi's monotone convergence theorem guarantees
\begin{equation*}
\lim_{n\to\infty} \int_{\R^N} w\left(u_n+\frac{1}{n}\right)^{-\gamma}\varphi_+ \dx = \int_{\R^N} wu^{-\gamma}\varphi_+ \dx.
\end{equation*}
In the same way,
\begin{equation*}
\lim_{n\to\infty} \int_{\R^N} w\left(u_n+\frac{1}{n}\right)^{-\gamma}\varphi_- \dx = \int_{\R^N} wu^{-\gamma}\varphi_- \dx,
\end{equation*}
producing
\begin{equation}
\label{sublimit2}
\lim_{n\to\infty} \int_{\R^N} w\left(u_n+\frac{1}{n}\right)^{-\gamma}\varphi \dx = \int_{\R^N} wu^{-\gamma}\varphi \dx.
\end{equation}

Hence \eqref{sublimit1}--\eqref{sublimit2} ensure that $u$ solves \eqref{subprob}, once we prove that $u\to 0$ as $|x|\to+\infty$. To do this, notice that $-\Delta_p u \geq 0$ in $B_R^e$, where $R>0$ stems from \eqref{weightdecay}. Given $\sigma>0$, set $\Phi_\sigma(x):=\sigma|x|^{\frac{p-N}{p-1}}$ and observe that $-\Delta_p \Phi_\sigma = 0$ in $B_R^e$. According to the weak comparison principle in exterior domains (see, e.g., \cite[Proposition A.12]{GG2}), there exists $\sigma>0$ small enough such that $u\geq \Phi_\sigma$ in $B_R^e$. This, together with \ref{hypw}, implies that $u$ solves
\begin{equation*}
-\Delta_p u \leq c_3\sigma^{-\gamma}|x|^{\gamma\,\frac{N-p}{p-1}-l} \quad \mbox{in}\;\; B_R^e.
\end{equation*}
Choosing $M>u(R)$ large enough, \cite[Lemma 3.2]{GG2} ensures the existence of $\Psi\in\D^{1,p}_0(\R^N)$ solving
\begin{equation*}
\left\{ \begin{alignedat}{2}
-\Delta_p \Psi &= c_3\sigma^{-\gamma}|x|^{\gamma\,\frac{N-p}{p-1}-l} \quad &&\mbox{in} \;\; B_R^e, \\
\Psi &= M \quad &&\mbox{on} \;\; \partial B_R^e, \\
\Psi(x) &\to 0 \quad &&\mbox{as} \;\; |x|\to+\infty.
\end{alignedat}
\right.
\end{equation*}
Thus, applying the aforementioned weak comparison principle to $u$ and $\Psi$, we infer $u\leq \Psi$ in $\R^N$; in particular $u(x)\to 0$ as $|x|\to+\infty$.

Recalling that $u\geq u_1$ in $\R^N$ and $\inf_{B_r} u_1>0$ for all $r>0$, we deduce $wu^{-\gamma}\in L^\infty_\loc(\R^N)$, implying that $u\in C^{1,\alpha}_\loc(\R^N)$ by regularity theory (see \cite[Corollary p.830]{DB}). On the other hand, exploiting the fact that $wu^{-\gamma} \leq c_3 \sigma^{-\gamma} |x|^{\gamma\,\frac{N-p}{p-1}-l}$ in $B_R^e$ and $l>N+\gamma\,\frac{N-p}{p-1}$, we infer $wu^{-\gamma}\in L^1(B_R^e)\cap L^\infty(B_R^e)$. Summarizing, $wu^{-\gamma}\in L^1(\R^N) \cap L^\infty(\R^N)$.

Uniqueness of $u$ can be proved by comparison as follows. If $u_1,u_2\in \D^{1,p}_0(\R^N)$ are two solutions of \eqref{subprob} then Lemma \ref{weakcomp} yields
$u_1\leq u_2$ in $\R^N$, since $wu_1^{-\gamma}<wu_2^{-\gamma}$ on $\{u_1>u_2\}$. Reversing the roles of $u_1$ and $u_2$ leads to $u_2\leq u_1$ in $\R^N$, whence $u_1=u_2$ in $\R^N$.
\end{proof}

\section{A truncated, frozen problem}\label{trfr}

Hereafter we will tacitly suppose \ref{hypf}--\ref{hypw}. Let $\underline{u}_\lambda\in C^{1,\alpha}_\loc(\R^N)$ be the solution to
\begin{equation*}
\left\{ \begin{alignedat}{2}
-\Delta_p u &=\lambda c_1 w(x) u^{-\gamma} \quad &&\mbox{in} \;\; \R^N, \\
u &> 0 \quad &&\mbox{in} \;\; \R^N, \\
u(x) &\to 0 \quad &&\mbox{as} \;\; |x|\to+\infty,
\end{alignedat}
\right.
\end{equation*}
whose existence and uniqueness are guaranteed by Lemma \ref{subsol}, replacing $w$ with $\lambda c_1 w$. According to the scaling properties of the $p$-Laplacian and the singularity $u^{-\gamma}$, we have 
\begin{equation}\label{usub}
\underline u_\lambda = \lambda^{\frac{1}{p-1+\gamma}}\underline{u},
\end{equation}
where $\underline{u}$ solves \eqref{subprob} with $c_1w$ in place of $w$. For any fixed $v\in\D^{1,p}_0(\R^N)$, set
\begin{equation}
\label{a}
a(x,s):=w(x)f(\max\{s,\underline{u}_\lambda(x)\},\nabla v(x)) \quad \forall (x,s)\in \R^N\times\R
\end{equation}
and consider the `truncated and frozen' problem
\begin{equation}
\label{varprob}
\tag{${\rm \hat{P}}_\lambda$}
-\Delta_p u = \lambda a(x,u) + u_+^{p^*-1} \quad \mbox{in} \;\; \R^N.
\end{equation}

Problem \eqref{varprob} has variational structure: its energy functional is
\begin{equation}
\label{Jdef}
J:\D^{1,p}_0(\R^N)\to\R, \quad J(u) := \frac{1}{p} \|\nabla u\|_p^p -\lambda\int_{\R^N} A(x,u) \dx - \frac{1}{p^*} \|u_+\|_{p^*}^{p^*} \quad \forall u\in\D^{1,p}_0(\R^N),
\end{equation}
being $A(x,s):=\int_0^s a(x,t) \dt$. Analogously we define $a_n,A_n,J_n$, replacing $v$ with $v_n\in\D^{1,p}_0(\R^N)$.

By \ref{hypf} we deduce the following estimates, valid for all $(x,s)\in\R^N\times\R$:
\begin{equation}
\label{aest}
c_1w(x)\max\{s,\underline u_\lambda(x)\}^{-\gamma}\leq a(x,s)\leq c_2 w(x)(\max\{s,\underline u_\lambda(x)\}^{-\gamma}+|\nabla v(x)|^{r-1}),
\end{equation}
\begin{equation}
\label{Aest}
\begin{split}
A(x,s) &\geq -c_2w(x)\left(\underline{u}_\lambda(x)^{-\gamma}+|\nabla v(x)|^{r-1}\right)|s|, \\
A(x,s) &\leq \frac{c_2}{1-\gamma}w(x) |s|^{1-\gamma} + c_2w(x)\left(\underline{u}_\lambda(x)^{-\gamma}+|\nabla v(x)|^{r-1}\right)|s|.
\end{split}
\end{equation}

In the sequel, we will make use of $\zeta,\theta\in(1,+\infty)$ defined as
$$ \frac{1}{\zeta}+\frac{1-\gamma}{p^*}=1, \quad \frac{1}{\theta}+\frac{r-1}{p}+\frac{1}{p^*}=1. $$

According to \eqref{Aest}, it is readily seen that $J$ is well-defined and of class $C^1$, with
$$\langle J'(u),\varphi \rangle = \int_{\R^N} |\nabla u|^{p-2}\nabla u\nabla \varphi \dx - \lambda\int_{\R^N} a(x,u)\varphi \dx - \int_{\R^N}u_+^{p^*-1}\varphi \dx \quad \forall u,\varphi\in\D^{1,p}_0(\R^N).$$

First, we prove two general results concerning, in particular, Palais-Smale sequences associated with the functionals $J_n$: Lemma \ref{enestlemma} provides an energy estimate, whilst Lemma \ref{ccplemma} detects a `critical' energy level under which compactness is recovered.

\begin{lemma}[Energy estimate]
\label{enestlemma}
Suppose $\lambda\in(0,1]$. Let $c\in\R$, $L>0$, and $(u_n),(v_n)\subseteq\D^{1,p}_0(\R^N)$ such that
\begin{equation}
\label{generalhyps}
\begin{aligned}
&\limsup_{n\to\infty} J_n(u_n)\leq c, \\
&\lim_{n\to\infty} J_n'(u_n) = 0 \quad \mbox{in} \;\; \D^{-1,p'}(\R^N), \\
&\limsup_{n\to\infty} \|\nabla v_n\|_p \leq L.
\end{aligned}
\end{equation}
Then there exists $\hat{C}=\hat{C}(p,N,w,\gamma,r,c_1,c_2)>0$ such that
\begin{equation*}
\limsup_{n\to\infty} \|\nabla u_n\|_p^p \leq 2N\left[\hat C\lambda^{\frac{p}{p-1+\gamma}} \left(L^{p'(r-1)}+1\right)+c \right].
\end{equation*}
\end{lemma}
\begin{proof}
According to \eqref{generalhyps} we have
\begin{equation}
\label{enest:start}
\begin{split}
c+o(1)(1+\|\nabla u_n\|_p) &\geq J_n(u_n)-\frac{1}{p^*} \langle J'_n(u_n),u_n \rangle \\
&=\frac{1}{N}\|\nabla u_n\|_p^p-\lambda\int_{\R^N} A_n(x,u_n) \dx + \frac{\lambda}{p^*} \int_{\R^N} a_n(x,u_n)u_n \dx.
\end{split}
\end{equation}
By means of \eqref{aest}--\eqref{Aest} and \eqref{usub}, as well as H\"older's and Sobolev's inequalities, besides taking into account that $\lambda\in(0,1)$, we get
\begin{equation}
\label{enest:A}
\begin{split}
&\int_{\R^N} A_n(x,u_n)\dx \\
&\leq \frac{c_2}{1-\gamma} \int_{\R^N} w|u_n|^{1-\gamma}\dx + c_2 \int_{\R^N} w(\underline{u}_\lambda^{-\gamma}+|\nabla v_n|^{r-1})|u_n| \dx \\
&\leq \frac{c_2}{1-\gamma} \|w\|_\zeta \|u_n\|_{p^*}^{1-\gamma} + c_2 (\lambda^{-\frac{\gamma}{p-1+\gamma}}\|w\underline{u}^{-\gamma}\|_{(p^*)'} + \|w\|_\theta \|\nabla v_n\|_p^{r-1}) \|u_n\|_{p^*} \\
&\leq \frac{c_2}{1-\gamma}S^{-\frac{1-\gamma}{p}} \|w\|_\zeta \|\nabla u_n\|_p^{1-\gamma} + c_2 \lambda^{-\frac{\gamma}{p-1+\gamma}} S^{-\frac{1}{p}} (\|w\underline{u}^{-\gamma}\|_{(p^*)'} + \|w\|_\theta L^{r-1}+o(1)) \|\nabla u_n\|_p
\end{split}
\end{equation}
and
\begin{equation}
\label{enest:a}
\begin{split}
\int_{\R^N} a_n(x,u_n)u_n \dx &\geq - \int_{\R^N} a_n(x,u_n)|u_n| \dx \\
&\geq -c_2 \int_{\R^N} w(\underline{u}_\lambda^{-\gamma}+|\nabla v_n|^{r-1})|u_n| \dx \\
&\geq - c_2 \lambda^{-\frac{\gamma}{p-1+\gamma}} S^{-\frac{1}{p}} (\|w\underline{u}^{-\gamma}\|_{(p^*)'} + \|w\|_\theta L^{r-1}+o(1)) \|\nabla u_n\|_p.
\end{split}
\end{equation}
Inserting \eqref{enest:A}--\eqref{enest:a} into \eqref{enest:start} produces
\begin{equation*}
\begin{split}
c+o(1)(1+\|\nabla u_n\|_p) \geq &\frac{1}{N}\|\nabla u_n\|_p^p - \lambda\frac{c_2}{1-\gamma}S^{-\frac{1-\gamma}{p}} \|w\|_\zeta \|\nabla u_n\|_p^{1-\gamma} \\
&-\lambda^{\frac{p-1}{p-1+\gamma}} c_2 S^{-\frac{1}{p}} \left(1+\frac{1}{p^*}\right) (\|w\underline{u}^{-\gamma}\|_{(p^*)'} + \|w\|_\theta L^{r-1}+o(1)) \|\nabla u_n\|_p.
\end{split}
\end{equation*}
Re-absorbing the terms $\|\nabla u_n\|_p$ on the right via Peter-Paul's inequality\footnote{It is a well-known refined version of Young's inequality: given $\sigma>0$ and $p_1,\ldots,p_k\in(1,+\infty)$ such that $\sum_{j=1}^k p_j^{-1} = 1$, there exists $C_\sigma>0$ such that
$$ \prod_{j=1}^k b_j \leq \sigma \prod_{j=1}^{k-1} b_j^{p_j} + C_\sigma b_k^{p_k} \quad \mbox{for all} \;\; b_1,\ldots,b_k\geq 0.$$}, we get
\begin{equation}
\label{enest:mid}
\left(\frac{1}{N}-\eps\right)\|\nabla u_n\|_p^p \leq C_\eps\lambda^{\frac{p}{p-1+\gamma}}(L^{p'(r-1)}+1)+c+o(1),
\end{equation}
for a suitable $C_\eps=C_\eps(\eps,p,N,w,\gamma,r,c_1,c_2)$ (since $S$ is a function of only $p,N$ and $\underline{u}$ depends uniquely on $p,w,\gamma,c_1$). Choosing $\eps=\frac{1}{2N}$ and setting $\hat{C}:=C_\eps$, \eqref{enest:mid} becomes
\begin{equation}
\label{enest:end}
\frac{1}{2N}\|\nabla u_n\|_p^p \leq \hat C\lambda^{\frac{p}{p-1+\gamma}} \left(L^{p'(r-1)}+1\right)+c+o(1).
\end{equation}
Passing to $\limsup$ for $n\to\infty$ concludes the proof.
\end{proof}

\begin{lemma}[Concentration compactness]
\label{ccplemma}
Assume the hypotheses of Lemma \ref{enestlemma}. If
\begin{equation}
\label{hatc}
c<\frac{S^{N/p}}{N}-\hat C\lambda^{\frac{p}{p-1+\gamma}} \left(L^{p'(r-1)}+1\right)=:\hat{c},
\end{equation}
then there exists $u\in \D^{1,p}_0(\R^N)$ such that, up to sub-sequences, $u_n\rightharpoonup u$ in $\D^{1,p}_0(\R^N)$ and
\begin{equation}
\label{strongconv}
(u_n)_+\to u_+ \quad \mbox{in} \;\; L^{p^*}(\R^N).
\end{equation}
In particular, $((u_n)_+)$ is uniformly equi-integrable in $L^{p^*}(\R^N)$, i.e.,
\begin{equation}
\label{equiint}
\int_\Omega (u_n)_+^{p^*} \dx \to 0 \quad \mbox{as} \;\; |\Omega|\to 0, \quad \mbox{uniformly in} \;\; n\in\N.
\end{equation}
\end{lemma}
\begin{proof}
By Lemma \ref{enestlemma} we deduce that $(u_n)$ is bounded in $\D^{1,p}_0(\R^N)$. Hence there exists $u\in\D^{1,p}_0(\R^N)$ such that, up to sub-sequences, $u_n \rightharpoonup u$ in $\D^{1,p}_0(\R^N)$, which implies $(u_n)_+ \rightharpoonup u_+$ in $\D^{1,p}_0(\R^N)$ by Lemma \ref{posparts}. In particular, $\nabla (u_n)_+ \rightharpoonup \nabla u_+$ in $L^p(\R^N)$, ensuring that the sequence of measures $(|\nabla (u_n)_+|^p)$ is bounded. Thus, $|\nabla (u_n)_+|^p \rightharpoonup \mu$ for some bounded measure $\mu$; analogously, $(u_n)_+^{p^*} \rightharpoonup \nu$ for an opportune bounded measure $\nu$.
According to Lemmas \ref{lions}--\ref{bennaoum}, applied to $((u_n)_+)$ in place of $(u_n)$, there exist some at most countable set $\A$, a family of points $(x_j)_{j\in\A} \subseteq \R^N$, and two families of numbers $(\mu_j)_{j\in\A},(\nu_j)_{j\in\A}\subseteq (0,+\infty)$ satisfying
\begin{equation}
\label{PS:measprops}
\begin{split}
\nu=u_+^{p^*}+\sum_{j\in\A} \nu_j \delta_{x_j}, &\quad \mu\geq |\nabla u_+|^p+\sum_{j\in\A} \mu_j \delta_{x_j}, \\
\limsup_{n\to\infty} \int_{\R^N} (u_n)_+^{p^*} \dx = \int_{\R^N} \, {\rm d}\nu + \nu_\infty, &\quad \limsup_{n\to\infty} \int_{\R^N} |\nabla (u_n)_+|^p \dx = \int_{\R^N} \, {\rm d}\mu+ \mu_\infty,
\end{split}
\end{equation}
and
\begin{equation}
\label{PS:measbounds}
S\nu_j^{p/p^*} \leq \mu_j \quad \mbox{for all} \;\; j\in\A, \quad S\nu_\infty^{p/p^*}\leq \mu_\infty,
\end{equation}
being
\begin{equation*}
\nu_\infty := \lim_{R\to+\infty} \limsup_{n\to\infty} \int_{B_R^e} (u_n)_+^{p^*} \dx, \quad \mu_\infty := \lim_{R\to\infty} \limsup_{n\to\infty} \int_{B_R^e} |\nabla (u_n)_+|^p \dx.
\end{equation*}

We claim that $\A=\emptyset$. By contradiction, let $j\in \A$ and $\psi\in C^\infty_c(\R^N)$ be a standard cut-off function fulfilling $\psi\equiv 1$ in $\overline{B}_{1/2}$, $\psi\equiv 0$ in $B_1^e$, and $0\leq \psi \leq 1$ in $\R^N$. For each $\eps\in(0,1)$ set
$$\psi_\eps(x):=\psi\left(\frac{x-x_j}{\eps}\right).$$

By hypothesis we have $\langle J_n'(u_n), (u_n)_+ \psi_\eps \rangle\to 0$, that is,
\begin{equation}
\label{PS:conctest}
\begin{split}
&\int_{\R^N} |\nabla (u_n)_+|^p \psi_\eps \dx + \int_{\R^N} (u_n)_+|\nabla u_n|^{p-2} \nabla u_n \nabla \psi_\eps \dx \\
&=\lambda \int_{\R^N} a_n(x,u_n)(u_n)_+\psi_\eps \dx +  \int_{\R^N} (u_n)_+^{p^*} \psi_\eps \dx + o(1).
\end{split}
\end{equation}

Since $\D^{1,p}_0(\R^N)\compact L^p(B_\eps(x_j))$, we have (up to sub-sequences) $u_n\to u$ both in $L^p(B_\eps(x_j))$ and in $\R^N$. Reasoning as for \eqref{enest:a}, besides recalling that $(|\nabla u_n|)$ is bounded in $L^p(\R^N)$, one has
\begin{equation}
\label{PS:conctest1}
\begin{aligned}
&\lim_{\eps \to 0} \lim_{n\to\infty} \left| \int_{\R^N} a_n(x,u_n)(u_n)_+\psi_\eps\dx \right| \\
&\leq  c_2 \lim_{\eps \to 0} \lim_{n\to\infty} \int_{B_\eps(x_j)} w(\underline{u}_\lambda^{-\gamma}+|\nabla v_n|^{r-1})|u_n| \dx \\
&\leq c_2\lambda^{-\frac{\gamma}{p-1+\gamma}} S^{-\frac{1}{p}} \left(\sup_{n\in\N} \|\nabla u_n\|_p\right) \lim_{\eps\to 0}\left(\|w\underline{u}^{-\gamma}\|_{L^{(p^*)'}(B_\eps(x_j))} + \|w\|_{L^\theta(B_\eps(x_j))}L^{r-1}\right) \\
&= 0.
\end{aligned}
\end{equation}
Since $\nabla \psi_\eps$ is bounded and compactly supported in $B_\eps(x_j)$, we deduce $u_n\nabla \psi_\eps \to u \nabla \psi_\eps$ in $L^p(\R^N)$ as $n\to\infty$. By H\"older's inequality and a change of variable we infer
\begin{equation}
\label{PS:conctest2}
\begin{split}
&\lim_{\eps \to 0} \lim_{n\to\infty} \left|\int_{\R^N} u_n |\nabla u_n|^{p-2} \nabla u_n \nabla \psi_\eps \dx\right| \\
&\leq \lim_{\eps \to 0} \lim_{n\to\infty} \|\nabla u_n\|_p^{p-1} \|u_n\nabla \psi_\eps\|_p \\
&\leq \left(\sup_{n\in\N} \|\nabla u_n\|_p\right)^{p-1} \lim_{\eps \to 0} \|u\nabla \psi_\eps\|_p \\
&\leq \left(\sup_{n\in\N} \|\nabla u_n\|_p\right)^{p-1} \|\nabla \psi\|_N \lim_{\eps \to 0} \|u\|_{L^{p^*}(B_\eps(x_j))} = 0.
\end{split}
\end{equation}
Passing to the limit in \eqref{PS:conctest} via \eqref{PS:conctest1}--\eqref{PS:conctest2} we get
\begin{equation*}
\lim_{\eps \to 0} \lim_{n\to\infty} \int_{\R^N} |\nabla (u_n)_+|^p \psi_\eps \dx = \lim_{\eps \to 0} \lim_{n\to\infty} \int_{\R^N} (u_n)_+^{p^*} \psi_\eps \dx.
\end{equation*}
Recalling Remark \ref{nujmuj}, we obtain
\begin{equation}
\label{PS:proportional}
\mu_j = \nu_j.
\end{equation}
Combining \eqref{PS:proportional} with \eqref{PS:measbounds} entails $S\nu_j^{p/p^*-1}\leq1$, whence
\begin{equation}
\label{PS:lowerbounds}
\mu_j=\nu_j \geq S^{N/p}.
\end{equation}
Now using \eqref{PS:lowerbounds} and arguing as in \eqref{enest:end} we have 
\begin{equation}
\label{PS:concfinal}\begin{aligned}
c+o(1)&\ge J_n(u_n)-\frac{1}{p^*}\langle J_n'(u_n),u_n\rangle
\\&
\ge\frac{1}{N}\|\nabla u_n\|_p^p+\lambda\left[\frac{1}{p^*}\int_{\R^N}a_n(x,u_n)u_n \dx-\int_{\R^N}A_n(x,u_n) \dx\right]\\
&\ge\frac{S^{N/p}}{N}+\frac{1}{N}\|\nabla u\|_p^p+\lambda\left[\frac{1}{p^*}\int_{\R^N}a_n(x,u)u \dx-\int_{\R^N}A_n(x,u) \dx\right]+o(1)\\
&\ge \frac{S^{N/p}}{N}+\frac{1}{2N}\|\nabla u\|_p^p-\hat C\lambda^{\frac{p}{p-1+\gamma}} \left(L^{p'(r-1)}+1\right)+o(1)\\
&\ge \frac{S^{N/p}}{N}-\hat C\lambda^{\frac{p}{p-1+\gamma}} \left(L^{p'(r-1)}+1\right)+o(1),
\end{aligned}\end{equation}
contradicting \eqref{hatc} as $n\to\infty$. This forces $\A=\emptyset$.

A similar argument proves that concentration cannot occur at infinity, that is, $\nu_\infty=\mu_\infty=0$: indeed, using a cut-off function $\psi_R\in C^\infty(\R^N)$ such that $\psi_R \equiv 0$ in $\overline{B}_R$, $\psi_R \equiv 1$ in $B_{2R}^e$, and $0\leq \psi_R\leq 1$ in $\R^N$, one arrives at $\nu_\infty = \mu_\infty$, and then the conclusion follows as in \eqref{PS:concfinal}. According to \eqref{PS:measprops}, besides $\A=\emptyset$ and $\nu_\infty=0$, we obtain
\begin{equation*}
\limsup_{n\to\infty} \int_{\R^N} (u_n)_+^{p^*} \dx = \int_{\R^N} u_+^{p^*} \dx,
\end{equation*}
that is \eqref{strongconv}, by uniform convexity (see \cite[Proposition 3.32]{B}); in turn, \eqref{strongconv} forces \eqref{equiint} by Vitali's convergence theorem (see, e.g., \cite[Corollary 4.5.5]{Bo}).
\end{proof}

Problem \eqref{varprob} exhibits double lack of compactness, due to the setting $\R^N$ and the presence of a reaction term with critical growth. The aim of the next lemma is to recover compactness on the energy levels that lie under the critical level $\hat{c}$ defined in \eqref{hatc}.

\begin{lemma}[Palais-Smale condition]
\label{PS}
Let $\lambda\in(0,1]$, $L>0$, and $v\in\D^{1,p}_0(\R^N)$ be such that $\|\nabla v\|_p\leq L$. Then $J$ satisfies the ${\rm(PS)}_c$ condition for all $c\in\R$ fulfilling \eqref{hatc}.
\end{lemma}
\begin{proof}
Take any $c\in\R$ as in \eqref{hatc} and consider an arbitrary ${\rm (PS)}_c$ sequence $(u_n)$ associated to the functional $J$. Applying Lemma \ref{ccplemma} with $v_n\equiv v$ and $J_n\equiv J$ entails $u_n\rightharpoonup u$ in $\D^{1,p}_0(\R^N)$ and $(u_n)_+\to u_+$ in $L^{p^*}(\R^N)$. Let us evaluate
\begin{equation}
\label{PS:S+start}
\begin{split}
\langle J'(u_n),u_n-u \rangle &= \int_{\R^N} |\nabla u_n|^{p-2}\nabla u_n\nabla (u_n-u)\dx -\lambda \int_{\R^N} a(x,u_n)(u_n-u) \dx \\
&\quad- \int_{\R^N} (u_n)_+^{p^*-1}(u_n-u) \dx.
\end{split}
\end{equation}
Since $\D^{1,p}_0(\R^N)\hookrightarrow L^{p^*}(\R^N)$, we have $u_n \rightharpoonup u$ in $L^{p^*}(\R^N)$ and $u_n\to u$ in $\R^N$, up to sub-sequences (see, e.g., \cite[p.3044]{GG1}). Thus, observing that \eqref{strongconv} forces $(u_n)_+^{p^*-1}\to u_+^{p^*-1}$ in $L^{(p^*)'}(\R^N)$, we get
\begin{equation}
\label{PS:S+crit}
\lim_{n\to\infty} \int_{\R^N} (u_n)_+^{p^*-1}(u_n-u) \dx = 0. 
\end{equation}
Reasoning as for \eqref{enest:a} one has
$$ \|w(\underline{u}_\lambda^{-\gamma}+|\nabla v|^{r-1})\|_{(p^*)'} \leq \lambda^{-\frac{\gamma}{p-1+\gamma}}\|w\underline{u}^{-\gamma}\|_{(p^*)'}+\|w\|_\theta L^{r-1} < +\infty, $$
so the linear functional $\psi\mapsto \int_{\R^N} w(\underline{u}_\lambda^{-\gamma}+|\nabla v|^{r-1}) \psi \dx$ is continuous in $L^{p^*}(\R^N)$. Moreover, \cite[Proposition 1]{GG1} guarantees that $|u_n-u|\rightharpoonup 0$ in $L^{p^*}(\R^N)$. Accordingly, \ref{hypf} ensures
\begin{equation}
\label{PS:S+sing}
\lim_{n\to\infty} \left|\int_{\R^N} a(x,u_n)(u_n-u) \dx \right| \leq c_2 \lim_{n\to\infty} \int w(\underline{u}_\lambda^{-\gamma}+|\nabla v|^{r-1})|u_n-u| \dx = 0.
\end{equation}
Using \eqref{PS:S+crit}--\eqref{PS:S+sing} and recalling that $\langle J'(u_n),u_n-u \rangle \to 0$, due to the fact that $(u_n)$ is a ${\rm (PS)}_c$ sequence, by \eqref{PS:S+start} we conclude
\begin{equation*}
\lim_{n\to\infty} \langle -\Delta_p u_n, u_n-u \rangle = \lim_{n\to\infty} \int_{\R^N} |\nabla u_n|^{p-2}\nabla u_n\nabla (u_n-u)\dx = 0.
\end{equation*}
Then the ${\rm (S_+)}$ property of $(-\Delta_p,\D^{1,p}_0(\R^N))$ (see \cite[Proposition 2.2]{MMM}) yields $u_n \to u$ in $\D^{1,p}_0(\R^N)$.
\end{proof}

The next two lemmas are devoted to verify the mountain pass geometry for the functional $J$ and ensure the mountain pass level $c_M$ (see Theorem \ref{mountainpass}) lies below the critical Palais-Smale level $\hat{c}$ (see \eqref{hatc}), provided $\lambda$ is small enough.

\begin{lemma}[Mountain pass geometry]
\label{mountainpassgeometry}
Let $L>0$ and $v\in\D^{1,p}_0(\R^N)$ be such that $\|\nabla v\|_p\leq L$. Then there exists $\Lambda_1\in(0,1)$ such that, for every $\lambda\in(0,\Lambda_1)$, the functional $J$ satisfies the mountain pass geometry. More precisely, there exists $\tilde{C}=\tilde{C}(p,N,w,\gamma,r,c_1,c_2)>1$ such that, for all $\lambda$ satisfying 
\begin{equation}
\label{smallness}
0 < \lambda < \Lambda_1:=\left[\frac{S^{N/p}}{\tilde{C}N(S^{N/p^2}+1)(1+L^{r-1})}\right]^{\frac{p-1+\gamma}{p-1}},
\end{equation}
one has
$$ J(\hat{u})<J(0)=0<\inf_{\partial B_\rho} J \quad \mbox{and} \quad \|\nabla \hat{u}\|_p > \rho, \quad \mbox{being} \;\; \rho:=S^{N/p^2},$$
for any $\hat{u}\in\D^{1,p}_0(\R^N)$ such that $\hat{u}_+ \not\equiv 0$ and $\|\nabla \hat{u}\|_p$ is sufficiently large.
\end{lemma}

\begin{proof}
Fix $\tilde{C}=\tilde{C}(p,N,w,\gamma,r,c_1,c_2)>1$  such that
$$ \frac{c_2}{1-\gamma}S^{-\frac{1-\gamma}{p}}\|w\|_\zeta + c_2 S^{-\frac{1}{p}} \left(\|w\underline{u}^{-\gamma}\|_{(p^*)'}+\|w\|_\theta L^{r-1}\right) \leq \tilde{C}(1+L^{r-1}). $$
Let $\lambda$ fulfill \eqref{smallness} and pick any $t>0$. Reasoning as in \eqref{enest:A}, besides recalling the choice of $\tilde{C}$, we get
\begin{equation*}
\begin{aligned}
&\inf_{\partial B_t} J \\
&\geq \inf_{u\in \partial B_t} \left[ \frac{1}{p}\|\nabla u\|_p^p - \lambda\frac{c_2}{1-\gamma}\int_{\R^N} w |u|^{1-\gamma} \dx -\lambda c_2 \int_{\R^N} w\left(\underline{u}_\lambda^{-\gamma}+|\nabla v|^{r-1}\right)|u| \dx - \frac{1}{p^*} \|u\|_{p^*}^{p^*} \right] \\
&\geq \inf_{u\in \partial B_t} \left[ \frac{1}{p}\|\nabla u\|_p^p - \lambda\frac{c_2}{1-\gamma}S^{-\frac{1-\gamma}{p}}\|w\|_\zeta \|\nabla u\|_p^{1-\gamma} \right. \\
&\left. \quad - c_2 \lambda^{\frac{p-1}{p-1+\gamma}} S^{-\frac{1}{p}} \left(\|w\underline{u}^{-\gamma}\|_{(p^*)'}+\|w\|_\theta L^{r-1}\right)\|\nabla u\|_p - \frac{1}{p^*}S^{-\frac{p^*}{p}} \|\nabla u\|_p^{p^*} \right] \\
&= \frac{1}{p}t^p - \lambda\frac{c_2}{1-\gamma}S^{-\frac{1-\gamma}{p}}\|w\|_\zeta t^{1-\gamma} - c_2 \lambda^{\frac{p-1}{p-1+\gamma}} S^{-\frac{1}{p}} \left(\|w\underline{u}^{-\gamma}\|_{(p^*)'}+\|w\|_\theta L^{r-1}\right)t - \frac{1}{p^*}S^{-\frac{p^*}{p}} t^{p^*} \\
&\geq \frac{1}{p}t^p - \lambda\frac{c_2}{1-\gamma}S^{-\frac{1-\gamma}{p}}\|w\|_\zeta (t+1) -  c_2 \lambda^{\frac{p-1}{p-1+\gamma}} S^{-\frac{1}{p}} \left(\|w\underline{u}^{-\gamma}\|_{(p^*)'}+\|w\|_\theta L^{r-1}\right)(t+1) - \frac{1}{p^*}S^{-\frac{p^*}{p}} t^{p^*} \\
&\geq \frac{1}{p}t^p - \lambda^{\frac{p-1}{p-1+\gamma}}(t+1)\left[\frac{c_2}{1-\gamma}S^{-\frac{1-\gamma}{p}}\|w\|_\zeta  +  c_2 S^{-\frac{1}{p}} \left(\|w\underline{u}^{-\gamma}\|_{(p^*)'}+\|w\|_\theta L^{r-1}\right)\right] - \frac{1}{p^*}S^{-\frac{p^*}{p}} t^{p^*} \\
&\geq \frac{1}{p}t^p - \tilde{C}\lambda^{\frac{p-1}{p-1+\gamma}}(1+L^{r-1})(t+1) - \frac{1}{p^*}S^{-\frac{p^*}{p}} t^{p^*},
\end{aligned}
\end{equation*}
since $\|\nabla u\|_p=t$ whenever $u\in\partial B_t$. Let us consider the real-valued function $g:(0,+\infty)\to\R$ defined as
\begin{equation*}
g(t) :=\frac{1}{p} t^p -\eps (t+1) -\frac{1}{p^*}S^{-\frac{p^*}{p}} t^{p^*},
\end{equation*}
where $\eps:=\tilde{C}\lambda^{\frac{p-1}{p-1+\gamma}}(1+L^{r-1})$. Condition \eqref{smallness} forces $\eps<\frac{S^{N/p}}{N(S^{N/p^2}+1)}$, so that
$$g(S^{N/p^2})=\frac{S^{N/p}}{N}-\eps (S^{N/p^2}+1)>0.$$
Setting $\rho:=S^{N/p^2}$ we get
$$\inf_{\partial B_\rho} J \ge  g(\rho)>0.$$

Given any $u\in\D^{1,p}_0(\R^N)$ such that $u_+\not\equiv 0$, we have $J(tu)\to-\infty$ as $t\to+\infty$: indeed, according to \eqref{Aest} and the fact that $\|u_+\|_{p^*}>0$,
$$\limsup_{t\to+\infty} J(tu)\leq \lim_{t\to+\infty} \left[ \frac{t^p}{p}\|\nabla u\|_p^p +c_2\lambda t\int_{\R^N} w(\underline{u}_\lambda^{-\gamma}+|\nabla v|^{r-1})|u| \dx -\frac{t^{p^*}}{p^*}\|u_+\|_{p^*}^{p^*} \right] = -\infty.$$
Hence, setting $\hat{u}:=tu$, we have $J(\hat{u})<0<\inf_{\partial B_\rho} J$ and $\|\nabla \hat{u}\|_p>\rho$, provided $t$ is sufficiently large.
\end{proof}

\begin{lemma}\label{talentihatc}
Assume the hypotheses of Lemma \ref{mountainpassgeometry}. Then there exists $\Lambda_2\in(0,1)$ such that for all $\lambda\in(0,\Lambda_2)$ one has
$$\inf_{\phi\in\Phi} \sup_{t\in[0,1]} J(\phi(t))<c<\hat{c}, \quad \Phi:=\{\phi\in C^0([0,1];X): \, \phi(0)=0, \, \phi(1)=\hat{u}\}$$
for an opportune $c=c(\lambda,p,N,w,\gamma,c_1)>0$, where $\hat{c}$ is defined in \eqref{hatc} and $\hat{u}$ is the Talenti function
$$\hat{u}(x)= \left[\frac{N^{\frac{1}{p}}\left(\frac{N-p}{p-1}\right)}{1+|x|^{p'}}\right]^{\frac{N-p}{p}}.$$
\end{lemma}
\begin{proof}
Firstly, we notice that $\|\nabla \hat{u}\|_p^p=\|\hat{u}\|_{p^*}^{p^*}=S^{N/p}$ (see \cite{PW} for details). Next, we observe that the path $\phi(u):=t\hat{u}$, $t\in[0,1]$, belongs to $\Phi$, so
\begin{equation*}
\inf_{\phi\in\Phi} \sup_{t\in[0,1]} J(\phi(t)) \leq \sup_{t\in[0,1]} J(t\hat{u}) \leq \sup_{t\in[0,+\infty)} J(t\hat{u}).
\end{equation*}
Accordingly, let us compute the maximizer $\overline{t}$ of the function $t\mapsto J(t\hat{u})$, being $t\geq 0$.
\begin{equation*}
0 = \frac{{\rm d}}{{\rm d}t} (J(t\hat{u}))_{\mid_{t=\overline{t}}} = \langle J'(\overline{t}\hat{u}),\hat{u} \rangle = \overline{t}^{p-1}\|\nabla \hat{u}\|_p^p - \lambda\int_{\R^N} a(x,\overline{t}\hat{u})\hat{u} \dx -  \overline{t}^{p^*-1} \|\hat{u}\|_{p^*}^{p^*},
\end{equation*}
whence
\begin{equation}
\label{lambdacrit}
\overline{t}^{p^*-1} \|\hat{u}\|_{p^*}^{p^*} = \overline{t}^{p-1}\|\nabla \hat{u}\|_p^p -\lambda \int_{\R^N} a(x,\overline{t}\hat{u})\hat{u} \dx.
\end{equation}
From \eqref{lambdacrit} we deduce 
$$\overline{t}^{p^*-1} \|\hat{u}\|_{p^*}^{p^*} \le \overline{t}^{p-1}\|\nabla \hat{u}\|_p^p=\overline{t}^{p-1} \|\hat{u}\|_{p^*}^{p^*},$$
forcing $\overline{t}\in[0,1]$.

Fix any $\Lambda_2\in(0,1)$ and define 
$$\Upsilon_1:=\left\{\lambda\in(0,\Lambda_2): \, \overline{t}\in \left[\frac{1}{2},1\right]\right\}, \qquad \Upsilon_2:=\left\{\lambda\in(0,\Lambda_2): \, \overline{t}\in \left[0,\frac{1}{2}\right]\right\}.$$
Suppose $\lambda\in\Upsilon_1$. 
Exploiting the monotonicity of $A(x,\cdot)$ and
$$ \max_{t\in(0,+\infty)}\left(\frac{t^p}{p}- \frac{t^{p^*}}{p^*}\right)=\frac{1}{N}, $$
we get
\begin{equation}\label{MPlevel:Jest}
\begin{split}
J(\overline{t}\hat{u}) &= \frac{\overline{t}^p}{p}\|\nabla\hat{u}\|_p^p - \lambda\int_{\R^N} A(x,\overline{t}\hat{u}) \dx - \frac{\overline{t}^{p^*}}{p^*} \|\hat{u}\|_{p^*}^{p^*}\\
&=\left(\frac{\overline{t}^p}{p}- \frac{\overline{t}^{p^*}}{p^*} \right)\|\nabla\hat{u}\|_p^p - \lambda\int_{\R^N} A(x,\overline{t}\hat{u}) \dx\le \frac{\|\nabla\hat{u}\|_p^p}{N}-\lambda\int_{\R^N} A\left(x, \frac{\hat{u}}{2}\right) \dx.
\end{split}
\end{equation}
Let $m:=\inf_{B_\varrho(x_0)}\hat u>0$, being $x_0,\varrho$ as in \eqref{weightbelow}. Imposing
\begin{equation}\label{smallness2}
\Lambda_2\leq \left(\frac{m}{4\|\underline u\|_\infty}\right)^{p-1+\gamma},
\end{equation}
one has $\|\underline u_\lambda\|_\infty\leq \Lambda_2^{\frac{1}{p-1+\gamma}} \|\underline{u}\|_\infty\leq m/4$. Accordingly, by \eqref{aest} and \eqref{weightbelow},
$$\begin{aligned}
\int_{\R^N} A\left(x, \frac{\hat{u}}{2}\right) \dx&\ge \int_{\{\hat{u}/2>\underline u_\lambda\}} \left(\int_{\underline u_\lambda}^{\hat{u}/2} a(x,t) \dt \right)\dx\ge c_1\int_{\{\hat{u}/2>\underline u_\lambda\}} \left(\int_{\underline u_\lambda}^{\hat{u}/2} w t^{-\gamma} \dt \right)\dx\\
&=\frac{c_1}{1-\gamma}\int_{\{\hat{u}/2>\underline u_\lambda\}} w(x)\left[\left(\frac{\hat u}{2}\right)^{1-\gamma}-\underline u_\lambda^{1-\gamma}\right]\dx\\
&\ge \frac{c_1}{1-\gamma}\int_{B_\varrho(x_0)} w(x)\left[\left(\frac{\hat u}{2}\right)^{1-\gamma}-\underline u_\lambda^{1-\gamma}\right]\dx\\
&\ge \frac{c_1}{1-\gamma}\left[\left(\frac{m}{2}\right)^{1-\gamma}-\left(\frac{m}{4}\right)^{1-\gamma}\right]\int_{B_\varrho(x_0)} w(x)\dx\\
&\ge \frac{c_1}{1-\gamma}\left[\left(\frac{m}{2}\right)^{1-\gamma}-\left(\frac{m}{4}\right)^{1-\gamma}\right] \omega |B_\varrho|=:2 \check C,
\end{aligned}$$
since $B_\varrho(x_0)\subseteq\{\hat{u}/2>\underline u_\lambda\}$. Incidentally, notice that $\check C$ depends only on $p,N,w,\gamma,c_1$. From \eqref{MPlevel:Jest} we deduce
\begin{equation}\label{Jest1}
J(\overline{t}\hat{u})\le \frac{S^{N/p}}{N}-2\lambda \check C.
\end{equation}
Now assume $\lambda\in \Upsilon_2$. Setting
$$ \check{c}:=\max_{t\in\left[0,\frac{1}{2}\right]}\left(\frac{t^p}{p}- \frac{t^{p^*}}{p^*}\right)<\frac{1}{N}, $$
we deduce
\begin{equation}\label{Jest2}
\begin{split}
J(\overline{t}\hat{u}) &= \frac{\overline{t}^p}{p}\|\nabla\hat{u}\|_p^p - \lambda\int_{\R^N} A(x,\overline{t}\hat{u}) \dx - \frac{\overline{t}^{p^*}}{p^*} \|\hat{u}\|_{p^*}^{p^*}\le\left(\frac{\overline{t}^p}{p}- \frac{\overline{t}^{p^*}}{p^*} \right)\|\nabla\hat{u}\|_p^p \le \check c S^{N/p}.
\end{split}\end{equation}
Recalling that
$$\hat{c} =\frac{S^{N/p}}{N}-\lambda^{\frac{p}{p-1+\gamma}} \hat C\left(1+L^{p'(r-1)}\right)$$
and noticing that $\frac{p}{p-1+\gamma}>1$, we can impose the additional bound
\begin{equation}\label{smallness3}
\Lambda_2\leq \min\left\{\frac{1-N\check c}{N\check C} \, S^{N/p},\left[\frac{\check C}{\hat C\left(1+L^{p'(r-1)}\right)}\right]^{\frac{p-1+\gamma}{1-\gamma}}\right\},
\end{equation}
so that
$$J(\overline{t}\hat{u})\leq \max\left\{\check c S^{N/p}, \frac{S^{N/p}}{N}-2\lambda\check C\right\} < \frac{S^{N/p}}{N}-\lambda\check C <\hat c$$
by \eqref{Jest1} and \eqref{Jest2}. The proof is concluded by choosing $c=\frac{S^{N/p}}{N}-\lambda\check C$.
\end{proof}


Now we are ready to prove existence of solutions to the truncated and frozen problem \eqref{varprob}.

\begin{thm}\label{existencefinal}
Let $L>0$ and $v\in\D^{1,p}_0(\R^N)$ be such that $\|\nabla v\|_p\leq L$. Suppose $\lambda\in(0,\Lambda)$ with $\Lambda:=\min\{\Lambda_1,\Lambda_2\}$, being $\Lambda_1$ and $\Lambda_2$ defined in Lemmas \ref{mountainpassgeometry} and \ref{talentihatc}, respectively. Then there exists $u\in \D^{1,p}_0(\R^N)$ solution to \eqref{varprob}.
\end{thm}

\begin{proof}
By virtue of Lemmas \ref{PS}, \ref{mountainpassgeometry}, and \ref{talentihatc}, the hypotheses of the mountain pass theorem (Theorem \ref{mountainpass}) are fulfilled: hence there exists $u\in \D^{1,p}_0(\R^N)$ solution to \eqref{varprob}.
\end{proof}

\begin{rmk}
\label{subcomparisonrmk}
Any solution $u$ to either \eqref{varprob} or \eqref{prob} satisfies $u\geq \underline{u}_\lambda$, being $\underline{u}_\lambda$ defined in \eqref{usub}. Indeed, if $u$ solves \eqref{varprob}, then it satisfies (in weak sense)
\begin{equation*}
-\Delta_p u \geq\lambda a(\cdot,u) = \lambda wf(\underline{u}_\lambda,\nabla v) \geq \lambda c_1w\underline{u}_\lambda^{-\gamma} = -\Delta_p \underline{u}_\lambda \quad \mbox{on} \;\; \{u<\underline{u}_\lambda\},
\end{equation*}
because of \eqref{a} and \ref{hypf}. Thus, Lemma \ref{weakcomp} yields $u\geq \underline{u}_\lambda$ in $\R^N$. A similar argument holds for solutions to \eqref{prob}, after noticing that they are positive by definition.
\end{rmk}

\section{Unfreezing the convection term}\label{unfr}

Set
\begin{equation}
\label{Ldef}
L:=2S^{N/p}.
\end{equation}
Let $\lambda\in(0,\Lambda)$ with $\Lambda=\min\{\Lambda_1,\Lambda_2,\Lambda_3\}$, where $\Lambda_1$ and $\Lambda_2$ stem from Lemmas \ref{mountainpassgeometry} (cf. \eqref{smallness}) and \ref{talentihatc} (cf. \eqref{smallness2} and \eqref{smallness3}), while $\Lambda_3$ will be determined in such a way that \eqref{lsccond2} below holds true.

Consider the $\D^{1,p}_0(\R^N)$-ball
$\B:=\{u\in\D^{1,p}_0(\R^N): \, \|\nabla u\|_p < L\}.$
Let $\S:\B \to \B$ be defined as
\begin{equation}
\label{Sdef}
\S(v) := \left\{u\in\D^{1,p}_0(\R^N): \, u \; \mbox{solves \eqref{varprob} and satisfies $J(u)< c$} \right\},
\end{equation}
where $c\in(0,\hat{c})$ stems from Lemma \ref{talentihatc}. We explicitly notice that $\S$ depends on $\lambda$; anyway, for the sake of simplicity, we omit this dependence.

\begin{lemma}
\label{welldefined}
The set-valued function $\S$ is well defined, i.e., $\S(\B)\subseteq \B$.
\end{lemma}
\begin{proof}
Take any $v\in\B$ and $u\in\S(v)$. Applying Lemma \ref{enestlemma} with $u_n\equiv u$, $v_n\equiv v$, and $J_n\equiv J$, after observing that $J(u)<c<\hat{c}$ and $J'(u)=0$ by definition of $\S$, we get
$$ \|\nabla u\|_p^p < 2N\left[ \hat{C}\lambda^{\frac{p}{p-1+\gamma}}\left(L^{p'(r-1)}+1\right)+\hat{c}\right]. $$
The conclusion then follows by recalling \eqref{hatc}.
\end{proof}

\begin{lemma}
\label{selectionlemma}
For any $v\in\B$, the set $\S(v)$ is non-empty and admits minimum.
\end{lemma}
\begin{proof}
Fix any $v\in\B$. The fact that $\S(v)\neq\emptyset$ is guaranteed by Theorem \ref{existencefinal}.

Now we prove that $\S(v)$ is downward directed. Let $u_1,u_2\in\S(v)$ and set $\overline{u}:=\min\{u_1,u_2\}$. Consider the truncation $T:\D^{1,p}_0(\R^N)\to \D^{1,p}_0(\R^N)$, $T(u)(x)=\tau(x,u(x))$, being
$\tau:\R^N\times \R\to \R$ defined as
$$\tau(x,t)=\left\{
\begin{alignedat}{2}
&\underline{u}_\lambda(x) \quad &&\mbox{if} \;\; t<\underline{u}_\lambda(x), \\
&t \quad &&\mbox{if} \;\; \underline{u}_\lambda(x)\leq t\leq \overline{u}(x), \\
&\overline{u}(x) \quad &&\mbox{if} \;\; t>\overline{u}(x).
\end{alignedat}\right.$$

We claim that there exists a solution $\check u\in\D^{1,p}_0(\R^N)$ to
\begin{equation}
\label{subprob1}
-\Delta_p u = \lambda a(x,T(u)) + (T(u))^{p^*-1} \quad \mbox{in} \;\; \R^N
\end{equation}
satisfying $J(\check u)<\hat{c}$.

The energy functional associated to \eqref{subprob1} is
$$\hat J(u)=\frac{1}{p}\|\nabla u\|_p^p-\int_{\R^N} \hat B(x,u) \dx,$$
where
$$\hat B(x,s):=\int_0^s \hat b(x,t) \dt, \qquad \hat b(x,t):=\lambda a(x,\tau(x,t))+\tau(x,t)^{p^*-1}.$$
From \eqref{aest}, \eqref{usub}, and $\lambda\in(0,1)$ we estimate
\begin{equation*}
\begin{aligned}
\hat b(x,t) &\leq c_2\lambda w(x)\left(\tau(x,t)^{-\gamma}+|\nabla v(x)|^{r-1}\right)+ \tau(x,t)^{p^*-1} \\
&\leq c_2\lambda w(x)\left(\underline{u}_\lambda(x)^{-\gamma}+|\nabla v(x)|^{r-1}\right)+ \tau(x,t)^{p^*-1} \\
&\leq c_2\lambda^{\frac{p-1}{p-1+\gamma}}w(x)\left(\underline{u}(x)^{-\gamma}+|\nabla v(x)|^{r-1}\right)+ \overline{u}(x)^{p^*-1}=:h(x),
\end{aligned}
\end{equation*}
so $h\in L^{(p^*)'}(\R^N)$. Hence $\hat J$ is coercive: indeed,
$$\hat J(u)\geq \frac{1}{p}\|\nabla u\|_p^p-\int_{\R^N} h(x)|u| \dx \geq \frac{1}{p}\|\nabla u\|_p^p-\|h\|_{(p^*)'}\|u\|_{p^*} \geq \frac{1}{p}\|\nabla u\|_p^p-S^{-\frac{1}{p}}\|h\|_{(p^*)'}\|\nabla u\|_{p}.$$
Moreover, it is readily seen that $\hat J$ is weakly sequentially lower semi-continuous. Thus, applying the direct methods of Calculus of Variations (see \cite[Theorem I.1.2]{S}), there exists $\check u\in \D^{1,p}_0(\R^N)$ such that $\hat J(\check u)=\min_{\D^{1,p}_0(\R^N)}\hat J$. In particular, $\hat J(\check u)\leq \hat J(0)=0$.

Reasoning as in the proof of \cite[Lemma 2.5.4]{G} (see also \cite[Lemma 3.4]{GMMot}), the minimum of two super-solutions to \eqref{varprob} is a super-solution to \eqref{varprob}; in particular, $\overline{u}$ is a super-solution to \eqref{varprob}. Since $\underline u_\lambda$ and $\overline{u}$ are respectively sub- and super-solution to \eqref{varprob}, Lemma \ref{weakcomp} ensures $\underline u_\lambda \leq \check u\leq \overline{u}$ in $\R^N$. Accordingly, $\check{u}$ solves \eqref{varprob} and $\hat J(\check u)= J(\check{u})$, so that $J(\check u)\leq 0 < c$.
The claim is proved. In addition, we obtained $\check u\in\S(v)$.

By arbitrariness of $u_1$ and $u_2$, the set $\S(v)$ is downward directed. Arguing as in \cite[Theorem 2.5.7]{G} (see also \cite[Lemma 3.14]{GMMot}), we conclude that $\S(v)$ admits minimum.
\end{proof}

\begin{lemma}
\label{Scompact}
The set-valued function $\S$ is compact.
\end{lemma}
\begin{proof}
Let $(v_n)$ be a (bounded) sequence in $\B$. For any $n\in\N$, pick $u_n\in\S(v_n)$. Our aim is to prove that $u_n \to u$ in $\D^{1,p}_0(\R^N)$ for some $u\in\D^{1,p}_0(\R^N)$.

Remark \ref{subcomparisonrmk} ensures $u_n\geq 0$ in $\R^N$ for all $n\in\N$, so Lemma \ref{ccplemma} produces $u\in\D^{1,p}_0(\R^N)$ such that $u_n \rightharpoonup u$ in $\D^{1,p}_0(\R^N)$ and $u_n\to u$ in $L^{p^*}(\R^N)$. Notice that, for all $n\in\N$,
\begin{equation}
\label{S+test}
\begin{aligned}
0 = \langle J'_n(u_n),u_n-u \rangle &= \int_{\R^N} |\nabla u_n|^{p-2}\nabla u_n \nabla (u_n-u) \dx \\
&\quad-\lambda \int_{\R^N} a_n(x,u_n)(u_n-u) \dx - \int_{\R^N} u_n^{p^*-1}(u_n-u) \dx.
\end{aligned}
\end{equation}
Computations similar to the ones of \eqref{PS:conctest1} show that $(a_n(\cdot,u_n))$ is bounded in $L^{(p^*)'}(\R^N)$, and $(u_n^{p^*-1})$ enjoys the same property; hence
$$ \lim_{n\to\infty} \int_{\R^N} a_n(x,u_n)(u_n-u) \dx = \lim_{n\to\infty} \int_{\R^N} u_n^{p^*-1}(u_n-u) \dx = 0. $$
Letting $n\to\infty$ in \eqref{S+test} entails
$$ \lim_{n\to\infty} \langle -\Delta_p u_n, u_n-u \rangle = \lim_{n\to\infty} \int_{\R^N} |\nabla u_n|^{p-2}\nabla u_n \nabla (u_n-u) \dx = 0, $$
so that the ${\rm (S_+)}$ property of $(-\Delta_p,\D^{1,p}_0(\R^N))$ (see \cite[Proposition 2.2]{MMM}) yields $u_n \to u$ in $\D^{1,p}_0(\R^N)$.
\end{proof}

\begin{lemma}
\label{lsc}
If $\Lambda_3>0$ is sufficiently small, then $\S$ is lower semi-continuous.
\end{lemma}
\begin{proof}

Let $v_n\to v$ in $\D^{1,p}_0(\R^N)$ and $u\in\S(v)$. We have to construct a sequence $(u_n)\subseteq\D^{1,p}_0(\R^N)$ such that $u_n\in \S(v_n)$ for every $n\in\N$ and $u_n\to u$ in $\D^{1,p}_0(\R^N)$. To this aim, we consider the following family of problems, parameterized by indexes $n,m\in\N$ and defined by recursion on $m$:
\begin{equation}
\label{recprob}
\left\{
\begin{alignedat}{2}
-\Delta_p u_n^m &= \lambda a_n(x, u_n^{m-1}) + (u_n^{m-1})^{p^*-1} \quad &&\mbox{in} \;\; \R^N, \\
u_n^0 &= u \quad &&\mbox{for all} \;\; n\in\N.
\end{alignedat}
\right.
\end{equation}
By induction on $m\in\N$, problem \eqref{recprob} admits a unique solution $u_n^m\in\D^{1,p}_0(\R^N)$ for all $n,m\in\N$, according to Minty-Browder's theorem \cite[Theorem 5.16]{B}.

Fixed $R>0$, for every $g:\R^N\to\R$ we define $g_R:\R^N\to\R$ as $g_R(x)=g(Rx)$ for all $x\in\R^N$. By a change of variables, $\|g_R\|_q = R^{-N/q} \|g\|_q$ for all $q\geq 1$.

We want to determine $R,\Lambda_3>0$ such that the set $\{z_n^m: \, n,m\in\N\}$ is bounded in $\D^{1,p}_0(\R^N)$, being $z_n^m:=(u_n^m)_R$ for all $n,m\in\N$. We observe that $z_n^m$ solves
\begin{equation}
\label{blowup}
-\Delta_p z_n^m = R^p \left[\lambda a_n(Rx,z_n^{m-1})+(z_n^{m-1})^{p^*-1}\right].
\end{equation}
Testing \eqref{blowup} with $z_n^m$, besides using \eqref{aest}, \ref{hypw}, \eqref{usub}, $\lambda \in(0,\Lambda)$, H\"older's inequality, and the boundedness of $(v_n)$ in $\D^{1,p}_0(\R^N)$, produces
\begin{equation*}
\begin{aligned}
&\|\nabla z_n^m\|_p^p \\
&\leq R^p \left[c_2 \lambda \int_{\R^N} w_R \left((\underline{u}_\lambda)_R^{-\gamma} + |(\nabla v_n)_R|^{r-1}\right) z_n^m \dx + \int_{\R^N} (z_n^{m-1})^{p^*-1}z_n^m \dx \right] \\
&\leq R^p \left[ c_2 \lambda^{\frac{p-1}{p-1+\gamma}}\int_{\R^N} w_R \left(\underline{u}_R^{-\gamma} + |(\nabla v_n)_R|^{r-1}\right) z_n^m \dx + \|z_n^{m-1}\|_{p^*}^{p^*-1} \|z_n^m\|_{p^*} \right] \\
&\leq R^p \left[ c_2 \Lambda^{\frac{p-1}{p-1+\gamma}} \left( \|w_R\underline{u}_R^{-\gamma}\|_{(p^*)'} + \|w_R\|_\theta\|(\nabla v_n)_R\|_p^{r-1} \right) \|z_n^m\|_{p^*} + \|z_n^{m-1}\|_{p^*}^{p^*-1} \|z_n^m\|_{p^*} \right] \\
&\leq R^p S^{-\frac{1}{p}}\|\nabla z_n^m\|_p \left[ c_2 \Lambda^{\frac{p-1}{p-1+\gamma}} R^{-\frac{N}{(p^*)'}} \left( \|w\underline{u}^{-\gamma}\|_{(p^*)'} + \|w\|_\theta\|\nabla v_n\|_p^{r-1} \right) + S^{-\frac{p^*-1}{p}} \|\nabla z_n^{m-1}\|_p^{p^*-1} \right].
\end{aligned}
\end{equation*}
Setting $H=\left[c_2 S^{-\frac{1}{p}}\left(\|w\underline{u}^{-\gamma}\|_{(p^*)'} + \|w\|_\theta L^{r-1}\right)\right]^{\frac{1}{p-1}}$ we get
\begin{equation}
\label{rec}
\|\nabla z_n^m\|_p \leq H \Lambda^{\frac{1}{p-1+\gamma}} R^{1-\frac{N}{p}} + R^{p'} S^{-\frac{p^*}{p(p-1)}}\|\nabla z_n^{m-1}\|_p^{\frac{p^*-1}{p-1}}.
\end{equation}
Now we want to apply Lemma \ref{reclemma} to \eqref{rec}. First we estimate, via Lemma \ref{welldefined},
\begin{equation*}
\|\nabla z_n^0\|_p = \|(\nabla u)_R\|_p = R^{-\frac{N}{p}}\|\nabla u\|_p \leq R^{-\frac{N}{p}}L.
\end{equation*}
Hence the fist condition in \eqref{smallnessconds} is met provided
\begin{equation}
\label{lsccond1}
\frac{1}{2} \geq R^{p'} S^{-\frac{p^*}{p(p-1)}} \left(R^{-\frac{N}{p}}L\right)^{\frac{p^*-p}{p-1}} = S^{-\frac{p^*}{p(p-1)}} \left(\frac{L}{R}\right)^{\frac{p^*-p}{p-1}}.
\end{equation}
On the other hand, the second condition in \eqref{smallnessconds} fulfilled whenever
\begin{equation}
\label{lsccond2}
2^{\frac{1-p^*}{p-1}} > R^{p'} S^{-\frac{p^*}{p(p-1)}} \left(H \Lambda^{\frac{1}{p-1+\gamma}} R^{1-\frac{N}{p}}\right)^{\frac{p^*-p}{p-1}} = H^{\frac{p^*-p}{p-1}} \Lambda^{\frac{p^*-p}{(p-1)(p-1+\gamma)}} S^{-\frac{p^*}{p(p-1)}}.
\end{equation}
Choosing $R=R(p,N)>0$ sufficiently large and $\Lambda_3=\Lambda_3(p,N,w,\gamma,r,c_1,c_2)>0$ small enough, both conditions \eqref{lsccond1}--\eqref{lsccond2} are fulfilled. By virtue of Lemma \ref{reclemma}, after noticing that all the quantities appearing in \eqref{lsccond1}--\eqref{lsccond2} do not depend on $n$, we conclude that there exists $\hat{L}=\hat{L}(p,N,R,\Lambda)>0$ such that $\|\nabla z_n^m\|_p \leq \hat{L} $ for all $n,m\in\N$, which implies $\|\nabla u_n^m\|_p \leq R^{N/p}\hat{L}$ for all $n,m\in\N$.

Now we pass to the weak limit the double sequence $(u_n^m)_{n,m}$ with respect to each index separately: up to sub-sequences, there exists $(u_n),(u^m)\subseteq\D^{1,p}_0(\R^N)$ such that
\begin{equation}
\label{weaklimits}
\begin{aligned}
u_n^m &\rightharpoonup u^m \quad \mbox{in} \;\; \D^{1,p}_0(\R^N) \;\; \mbox{as} \;\; n\to\infty, \quad \forall m\in\N, \\
u_n^m &\rightharpoonup u_n \quad \mbox{in} \;\; \D^{1,p}_0(\R^N) \;\; \mbox{as} \;\; m\to\infty, \quad \forall n\in\N.
\end{aligned}
\end{equation}
Letting $n\to\infty$ in the weak formulation of \eqref{recprob}, it turns out that both $u^1$ and $u$ solve the problem
\begin{equation}
\label{limitprob}
\left\{
\begin{alignedat}{2}
-\Delta_p U &= \lambda a(x, u(x)) + u(x)^{p^*-1} \quad &&\mbox{in} \;\; \R^N, \\
U &\in \D^{1,p}_0(\R^N).
\end{alignedat}
\right.
\end{equation}
Since \eqref{limitprob} admits a unique solution by Minty-Browder's theorem, we deduce $u^1=u$. Reasoning inductively on $m\in\N$, it follows that $u^m=u$ for all $m\in\N$. Pick an arbitrary $\rho>0$. Since $\D^{1,p}_0(\R^N)\compact L^p(B_\rho)$, the convergences mentioned in \eqref{weaklimits} are strong in $L^p(B_\rho)$. Accordingly, the double limit lemma \cite[Proposition A.2.35]{GP} guarantees, up to sub-sequences,
$$ \lim_{n\to\infty} u_n = \lim_{n\to\infty} \lim_{m\to\infty} u_n^m = \lim_{m\to\infty} \lim_{n\to\infty} u_n^m = \lim_{m\to\infty} u^m = u \quad \mbox{in} \;\; L^p(B_\rho). $$
In particular, since $\rho$ was arbitrary, a diagonal argument ensures $u_n\to u$ in $\R^N$.

Now we prove that $u_n\in\S(v_n)$ for all $n\in\N$. Letting $m\to\infty$ in the weak formulation of \eqref{recprob} reveals that $u_n$ solves \eqref{varprob} with $v=v_n$, for all $n\in\N$. Reasoning as in Lemma \ref{ccplemma}, boundedness of $(u_n)$ in $\D^{1,p}_0(\R^N)$ allows us to assume $|\nabla u_n|^p \rightharpoonup \mu$ and $u_n^{p^*} \rightharpoonup \nu$ for some bounded measures $\mu,\nu$. According to Lemmas \ref{lions}--\ref{bennaoum}, there exist some at most countable set $\A$, a family of points $(x_j)_{j\in\A} \subseteq \R^N$, and two families of numbers $(\mu_j)_{j\in\A},(\nu_j)_{j\in\A}\subseteq(0,+\infty)$ such that
\begin{equation*}
\begin{split}
\nu=u^{p^*}+\sum_{j\in\A} \nu_j \delta_{x_j}, &\quad \mu\geq |\nabla u|^p+\sum_{j\in\A} \mu_j \delta_{x_j}, \\
\limsup_{n\to\infty} \int_{\R^N} u_n^{p^*} \dx = \int_{\R^N} \, {\rm d}\nu + \nu_\infty, &\quad \limsup_{n\to\infty} \int_{\R^N} |\nabla u_n|^p \dx = \int_{\R^N} \, {\rm d}\mu+ \mu_\infty,
\end{split}
\end{equation*}
and
\begin{equation}
\label{lsc:measbounds}
S\nu_j^{p/p^*} \leq \mu_j \quad \mbox{for all} \;\; j\in\A, \quad S\nu_\infty^{p/p^*}\leq \mu_\infty,
\end{equation}
being $\mu_\infty,\nu_\infty$ as in Lemma \ref{bennaoum}. Suppose by contradiction that $\A\neq\emptyset$, so that $\mu_j=\nu_j\geq S^{N/p}$ for some $j\in\A$, according to \eqref{lsc:measbounds}. A computation analogous to \eqref{PS:concfinal}, jointly with $u\in\S(v)$, ensures that
$$ c > J(u) = J(u)-\frac{1}{p^*}\langle J'(u),u \rangle \geq \hat{c} $$
being $\hat{c}$ defined by \eqref{hatc}, which contradicts $c<\hat{c}$. Hence concentration at points cannot occur; as in Lemma \ref{ccplemma}, a similar argument excludes concentration at infinity. We deduce $u_n\to u$ in $L^{p^*}(\R^N)$, which is the starting point of the proof of Lemma \ref{Scompact}; thus we infer $u_n\to u$ in $\D^{1,p}_0(\R^N)$. In particular, $J_n(u_n)\to J(u)$ as $n\to\infty$, so $J_n(u_n)<c$ for all $n$ sufficiently large, ensuring $u_n\in\S(v_n)$.

\end{proof}

\begin{thm}
\label{exsol}
For any $\lambda\in(0,\Lambda)$ the problem 
\begin{equation}
\label{probnotdecay}
\left\{
\begin{alignedat}{2}
-\Delta_p u &=\lambda w(x)f(u,\nabla u) +  u^{p^*-1} \quad &&\mbox{in} \;\; \R^N, \\
u &> 0 \quad &&\mbox{in} \;\; \R^N, \\
\end{alignedat}
\right.
\end{equation}
admits a solution $u\in\D^{1,p}_0(\R^N)$.
\end{thm}
\begin{proof}
Let us consider the following selection of the multi-function $\S$ defined in \eqref{Sdef}:
\begin{equation}
\label{Tdef}
\T:\B\to\B, \quad \T(v) = \min \S(v).
\end{equation}
The function $\T$ is well defined, according to Lemma \ref{selectionlemma}; moreover, it is continuous and compact, since $\S$ is lower semi-continuous and compact by Lemmas \ref{Scompact}--\ref{lsc} (see \cite[Lemma 3.16]{GMMot} for details). According to Schauder's theorem (Theorem \ref{schauder}), $\T$ admits a fixed point $u\in\D^{1,p}_0(\R^N)$. Remark \ref{subcomparisonrmk} guarantees $u\geq \underline{u}_\lambda$, so that $u$ solves \eqref{probnotdecay}.
\end{proof}

\begin{rmk}
We observe that the choice \eqref{Ldef} was made for the sake of simplicity: actually, any choice of a smaller $L>S^{N/p}$ allows to prove Theorem \ref{exsol}, provided $\Lambda$ is small enough. To see that, it suffices to perform the energy estimate and the estimate of $\hat{c}$ retaining $\eps$ in Lemmas \ref{enestlemma}--\ref{ccplemma} instead of setting $\eps=\frac{1}{2N}$. More precisely, for any $v\in\B$ and $u\in\S(v)$ the following estimates hold true:
\begin{equation*}
\begin{aligned}
\|\nabla u\|_p^p &<\left(\frac{1}{N}-\eps\right)^{-1} \left[ \hat{C}_\eps \lambda^{\frac{p}{p-1+\gamma}}\left(L^{p'(r-1)}+1\right)+\hat{c} \right], \\
\hat{c} &= \frac{S^{N/p}}{N} - \hat{C}_\eps \lambda^{\frac{p}{p-1+\gamma}}\left(L^{p'(r-1)}+1\right).
\end{aligned}
\end{equation*}
Thus,
$$ \|\nabla u\|_p^p < \left(\frac{1}{N}-\eps\right)^{-1} \frac{S^{N/p}}{N} \to S^{N/p} \quad \mbox{as} \;\; \eps\to 0, $$
ensuring the validity of Lemma \ref{welldefined}. Anyway, according to \eqref{smallness3}, one has $\Lambda_2\to 0$ as $\eps\to 0$, since $\hat{C}_\eps\to+\infty$: for this reason, the choice $L=S^{N/p}$ is not feasible. On the contrary, $L=S^{N/p}$ is admissible in the model case $\lambda=0$, even if concentration of compactness occurs.
\end{rmk}

\section{Regularity of solutions}\label{reg}

In this section we prove that any solution $u$ to \eqref{probnotdecay} lying in an energy level under the critical Palais-Smale level $\hat{c}$ (see \eqref{hatc}) belongs to $L^\infty(\R^N)\cap C^{1,\alpha}_\loc(\R^N)$, and the estimates are uniform with respect to $u$ within the energy level chosen. In addition, $u$ decays pointwise as $|x|\to+\infty$. We conclude the section with the proof of Theorem \ref{mainthm} and a remark concerning a problem related to \eqref{prob}.

\begin{thm}
\label{regularity}
Let $\lambda\in(0,1)$ and $u\in\D^{1,p}_0(\R^N)$ be a solution to \eqref{probnotdecay} satisfying
$$ J(u)<\hat{c},  $$
where $J$ and $\hat{c}$ are defined respectively in \eqref{Jdef} and \eqref{hatc}, with $v=u$ and $L:=\|\nabla u\|_p$. Then
$$\|u\|_\infty \leq M,$$
for an opportune $M=M(p,N,w,r,c_2)>0$, and
$$\|u\|_{C^{1,\alpha}(\overline{B}_R)} \leq C_R,$$
for some $C_R=C_R(R,p,N,w,\gamma,r,c_1,c_2)>0$ and $\alpha\in(0,1]$. Moreover,
\begin{equation}
\label{decay}
u(x)\to 0 \quad \mbox{as} \;\; |x|\to+\infty.
\end{equation}
\end{thm}

\begin{proof}
Given any $k>1$, we test \eqref{probnotdecay} with $(u-k)_+$. Recalling that $\lambda\in(0,1)$ and $u>k>1$ in $\Omega_k:=\{x\in\R^N: \, u(x)>k\}$, besides using \ref{hypf}--\ref{hypw} and Peter-Paul's inequality, we get
\begin{equation*}
\begin{aligned}
&\|\nabla (u-k)\|_{L^p(\Omega_k)}^p \\
&\leq \lambda \int_{\Omega_k} wf(u,\nabla u)u \dx + \|u\|_{L^{p^*}(\Omega_k)}^{p^*} \\
&\leq c_2 \left(\|w\|_\infty \int_{\Omega_k} u^{1-\gamma} \dx + \int_{\Omega_k} w|\nabla u|^{r-1}u \dx \right) + \|u\|_{L^{p^*}(\Omega_k)}^{p^*} \\
&\leq c_2\|w\|_\infty \left( \|u\|_{L^{p^*}(\Omega_k)}^{p^*} + \frac{\eps}{c_2\|w\|_\infty}\|\nabla u\|_{L^p(\Omega_k)}^p + C_\eps \|u\|_{L^{p^*}(\Omega_k)}^{p^*} + \|w\|_\infty^\theta |\Omega_k| \right) + \|u\|_{L^{p^*}(\Omega_k)}^{p^*} \\
&\leq \eps \|\nabla (u-k)\|_{L^p(\Omega_k)}^p + C_\eps\left(\|u\|_{L^{p^*}(\Omega_k)}^{p^*}+|\Omega_k|\right) \\
&\leq \eps \|\nabla (u-k)\|_{L^p(\Omega_k)}^p + C_\eps\left(\|u-k\|_{L^{p^*}(\Omega_k)}^{p^*}+k^{p^*}|\Omega_k|+|\Omega_k|\right).
\end{aligned}
\end{equation*}
Choosing $\eps=1/2$ and re-absorbing the term $\|\nabla (u-k)\|_{L^p(\Omega_k)}^p$ on the left-hand side, we get
$$ \|\nabla (u-k)\|_{L^p(\Omega_k)}^p \leq C\left(\|u-k\|_{L^{p^*}(\Omega_k)}^{p^*}+k^{p^*}|\Omega_k|\right), $$
for some $C>0$ depending only on $p,N,w,r,c_2$. By Sobolev's inequality we obtain
$$ \|u-k\|_{L^{p^*}(\Omega_k)}^p \leq C\left(\|u-k\|_{L^{p^*}(\Omega_k)}^{p^*}+k^{p^*}|\Omega_k|\right), $$
enlarging $C$ if necessary. Let $M>2$ and set $k_n:=M(1-2^{-n})$ for all $n\in\N$. Repeating verbatim the proof of \cite[Lemma 3.2]{CGL}, we infer that $\|u-k_n\|_{L^{p^*}(\Omega_{k_n})} \to 0$ as $n\to\infty$, provided $\|u-M/2\|_{L^{p^*}(\Omega_{M/2})}$ is small enough. According to Remark \ref{subcomparisonrmk}, $u$ solves \eqref{varprob} with $v=u$; hence Lemma \ref{ccplemma} provides $L^{p^*}(\R^N)$-integrability of $u$ which is uniform in the sub-level set $\{\tilde{u}\in\D^{1,p}_0(\R^N): J(\tilde{u})\leq J(u)<\hat{c}\}$. As a consequence (see \cite[p.267]{Bo}),
\begin{equation*}
\int_{\Omega_{M/2}} \left(u-\frac{M}{2}\right)^{p^*} \dx \leq \int_{\Omega_{M/2}} u^{p^*} \dx \to 0 \quad \mbox{as} \;\; M\to +\infty,
\end{equation*}
and the limit is uniform in $u$. Hence, $u\in L^\infty(\R^N)$ with uniform $L^\infty$ estimates.

Now we prove uniform $C^{1,\alpha}$ local estimates for $u$. By \ref{hypf}--\ref{hypw}, \eqref{usub}, Lemma \ref{subsol}, and the uniform $L^\infty$ estimate of $u$, we get
\begin{equation}
\label{holderreg2}
\begin{aligned}
0\leq \lambda f(u,\nabla u) + u^{p^*-1} &\leq c_2 \lambda \left(w\underline{u}_\lambda^{-\gamma} + w|\nabla u|^{r-1}\right) + u^{p^*-1} \\
&\leq  c_2\lambda^{\frac{p-1}{p-1+\gamma}}\|w\underline{u}^{-\gamma}\|_\infty+c_2\lambda \|w\|_\infty|\nabla u|^{r-1}+\|u\|_\infty^{p^*-1} \\
&= C(1+|\nabla u|^{r-1}),
\end{aligned}
\end{equation}
for a suitable $C=C(p,N,w,\gamma,r,c_1,c_2)>0$. Applying \cite[Theorem 1.5]{DM} (with $b(x,u,\nabla u)=1+|\nabla u|^{r-1}$, $V(x)\equiv C$, and $q=r-1$) yields $|\nabla u|\in L^\infty_{\rm loc}(\R^N)$ with uniform $L^\infty$ local estimates: more precisely,
\begin{equation}
\label{gradsupest}
\|\nabla u\|_{L^\infty(B_\rho)} \leq C \left(L + \rho^{\frac{1}{p-r+2}}\right) \quad \mbox{for all} \;\; \rho>0,
\end{equation}
enlarging $C$ if necessary. Thus, by \eqref{holderreg2}, the right-hand side of \eqref{prob} is locally bounded in $\R^N$ which entails, by nonlinear regularity theory (see \cite[Corollary p.830]{DB}), $u\in C^{1,\alpha}_{\rm loc}(\R^N)$ with uniform $C^{1,\alpha}$ local estimates.

In order to prove \eqref{decay}, we observe that \eqref{gradsupest} and \eqref{weightdecay} imply
$$ w(x)|\nabla u(x)|^{r-1} \leq C|x|^{\frac{r-1}{p-r+2}-l} \quad \mbox{in} \;\; B_R^e, $$
being $R$ as in \eqref{weightdecay}. Since
$$ l>N+\gamma\,\frac{N-p}{p-1}>N>\frac{N+1}{2}>\frac{p+1}{2}>\frac{p+1}{p-r+2}=\frac{r-1}{p-r+2}+1, $$
we deduce $\frac{r-1}{p-r+2}-l<-1$, whence $w|\nabla u|^{r-1}\in L^q(\R^N)$ for some $q>N$. Consequently, reasoning as for \eqref{holderreg2} leads to
$$ \sup_{y\in\R^N}\|\lambda f(u,\nabla u) + u^{p^*-1}\|_{L^q(B_2(y))} \leq c_2 \left(\lambda^{\frac{p-1}{p-1+\gamma}}\|w\underline{u}^{-\gamma}\|_q+\lambda\|w|\nabla u|^{r-1}\|_q\right) + \|u\|_\infty^{p^*-1} |B_2|^{1/q}.  $$
An application of \cite[Lemma 2.4]{GM} (adapting its proof to locally summable reaction terms) ensures that $\nabla u\in L^\infty(\R^N)$ uniformly in $u$. Hence, a computation similar to \eqref{holderreg2} yields
\begin{equation*}
\lambda f(u,\nabla u) + u^{p^*-1} \leq C(w\underline{u}_\lambda^{-\gamma}+1) \leq C|x|^{\gamma\,\frac{N-p}{p-1}-l} \quad \mbox{in} \;\; B_\rho^e,
\end{equation*}
for some $C=C(p,N,w,\gamma,r,c_1,c_2)>0$ and for all $\rho>0$ sufficiently large. The conclusion then follows by repeating the argument used in the proof of Lemma \ref{subsol}.
\end{proof}

\begin{proof}[Proof of Theorem \ref{mainthm}]
Theorem \ref{mainthm} is direct consequence of Theorems \ref{exsol} and \ref{regularity}.
\end{proof}

\begin{rmk}
At the end of the paper, we would like to get a glimpse of another problem, which is a critical perturbation of a singular problem:
\begin{equation}
\label{probbis}
\tag{${\rm P}'_\lambda$}
\left\{
\begin{alignedat}{2}
-\Delta_p u &= w(x)f(u,\nabla u) +  \lambda u^{p^*-1} \quad &&\mbox{in} \;\; \R^N, \\
u &> 0 \quad &&\mbox{in} \;\; \R^N, \\
\end{alignedat}
\right.
\end{equation}
being $\lambda\in(0,1)$ and $p,N,w,f$ as in \eqref{prob}.

Performing the change of variable $z=\lambda^{\frac{1}{p^*-p}}u$ one has
$$ -\Delta_p z = \lambda^{\frac{p-1}{p^*-p}} w(x) \tilde{f}(z,\nabla z) + z^{p^*-1}, $$
where
$$ \tilde{f}(s,\xi) = f\left(\lambda^{\frac{1}{p-p^*}}s,\lambda^{\frac{1}{p-p^*}}\xi\right).  $$
According to \ref{hypf}, one has
$$ c_1 \lambda^{\frac{p-1+\gamma}{p^*-p}}z^{-\gamma} \leq \lambda^{\frac{p-1}{p^*-p}} \tilde{f}(z,\nabla z) \leq c_2 \left( \lambda^{\frac{p-1+\gamma}{p^*-p}}z^{-\gamma} + \lambda^{\frac{p-r}{p^*-p}}|\nabla z|^{r-1} \right). $$
Since $\frac{p-1+\gamma}{p^*-p}>\frac{p-r}{p^*-p}$, problem \eqref{probbis} cannot be directly reduced to problem \eqref{prob}. Anyway, we notice that $\frac{p-1+\gamma}{p^*-p},\frac{p-r}{p^*-p}>0$, and this is mainly due to the $p$-sub-linearity of the convection terms; thus, the smallness conditions on $\lambda$ for problem \eqref{prob} are mapped into smallness conditions on $\lambda^{\frac{p-1+\gamma}{p^*-p}}$ and $\lambda^{\frac{p-r}{p^*-p}}$ for problem \eqref{probbis}. Accordingly, the techniques used in this paper may be effectively employed to study problem \eqref{probbis}.
\end{rmk}

\section*{Acknowledgments}
\noindent
We warmly thank Prof. Sunra Mosconi for his valuable comments about the mountain pass theorem and the $L^\infty$ estimates. \\
This work has been partially carried out during a stay at the Department of Mathematics and Computer Sciences of the University of
Catania: the authors would like to express their deep gratitude to this prestigious institution for its support and warm hospitality. \\
The authors are member of the {\em Gruppo Nazionale per l'Analisi Ma\-te\-ma\-ti\-ca, la Probabilit\`a e le loro Applicazioni}
(GNAMPA) of the {\em Istituto Nazionale di Alta Matematica} (INdAM); they are partially supported by the INdAM-GNAMPA Project 2023 titled {\em Problemi ellittici e parabolici con termini di reazione singolari e convettivi} (E53C22001930001). \\
Laura Baldelli is partially supported by National Science Centre, Poland \\
(Grant No. 2020/37/B/ST1/02742) and by the
``Maria de Maeztu'' Excellence Unit IMAG, reference CEX2020-001105-M, funded by MCIN/AEI/10.13039/501100011033/. \\
Umberto Guarnotta is supported by the following research projects: 1) PRIN 2017 `Nonlinear Differential Problems via Variational, Topological and Set-valued Methods' (Grant no. 2017AYM8XW) of MIUR; 2) `MO.S.A.I.C.' PRA 2020--2022 `PIACERI' Linea 3 of the University of Catania. \\
This study was carried out within the RETURN Extended Partnership and received funding from the European Union Next-GenerationEU (National Recovery and Resilience Plan – NRRP, Mission 4, Component 2, Investment 1.3 – D.D. 1243 2/8/2022, PE0000005).

\begin{small}

\end{small}

\end{document}